\documentclass[a4paper,12pt]{article}
\usepackage{latexsym}
\usepackage{amsmath}
\usepackage{amssymb}
\usepackage{enumerate}
\usepackage{epic,eepic}
\usepackage{mathrsfs}

\usepackage{theorem}
\newtheorem{theorem}{Theorem}[section]
\newtheorem{proposition}[theorem]{Proposition}

\newtheorem{corollary}[theorem]{Corollary}
\newtheorem{claim}[theorem]{Claim}
\newtheorem{question}[theorem]{Question}

\theorembodyfont{\rmfamily}
\newtheorem{proof}{\textmd{\textit{Proof.}}}

\newtheorem{oproof}{\textmd{\textit{Outline of proof.}}}

\newtheorem{remark}[theorem]{Remark}
\newtheorem{example}[theorem]{Example}
\newtheorem{definition}[theorem]{Definition}
\newtheorem{fread}{Further Reading}

\makeatletter

\@addtoreset{equation}{section}
\makeatother

\newcommand{\qedd}{\hfill \Box}
\newcommand{\ve}{\varepsilon}
\newcommand{\del}{\partial}
\newcommand{\lra}{\longrightarrow}
\renewcommand{\det}{\ensuremath{\mathrm{det}}}

\newcommand{\N}{\ensuremath{\mathbb{N}}}
\newcommand{\R}{\ensuremath{\mathbb{R}}}
\newcommand{\Z}{\ensuremath{\mathbb{Z}}}
\newcommand{\bbM}{\ensuremath{\mathbb{M}}}
\newcommand{\Sph}{\ensuremath{\mathbb{S}}}
\newcommand{\Hy}{\ensuremath{\mathbb{H}}}
\newcommand{\cA}{\ensuremath{\mathcal{A}}}
\newcommand{\cC}{\ensuremath{\mathcal{C}}}
\newcommand{\cF}{\ensuremath{\mathcal{F}}}
\newcommand{\cK}{\ensuremath{\mathcal{K}}}

\newcommand{\cP}{\ensuremath{\mathcal{P}}}
\newcommand{\cR}{\ensuremath{\mathcal{R}}}

\newcommand{\cU}{\ensuremath{\mathcal{U}}}

\newcommand{\bs}{\ensuremath{\mathbf{s}}}
\newcommand{\bA}{\ensuremath{\mathbf{A}}}

\newcommand{\bJ}{\ensuremath{\mathbf{J}}}
\newcommand{\bS}{\ensuremath{\mathbf{S}}}
\newcommand{\DC}{\ensuremath{\mathcal{DC}}}
\newcommand{\MCP}{\ensuremath{\mathsf{MCP}}}
\newcommand{\CD}{\ensuremath{\mathsf{CD}}}

\def\diam{\mathop{\mathrm{diam}}\nolimits}
\def\vol{\mathop{\mathrm{vol}}\nolimits}
\def\area{\mathop{\mathrm{area}}\nolimits}
\def\Hess{\mathop{\mathrm{Hess}}\nolimits}
\def\supp{\mathop{\mathrm{supp}}\nolimits}
\def\Ent{\mathop{\mathrm{Ent}}\nolimits}
\def\Ric{\mathop{\mathrm{Ric}}\nolimits}
\def\Id{\mathop{\mathrm{Id}}\nolimits}
\def\tr{\mathop{\mathrm{tr}}\nolimits}
\def\ac{\mathop{\mathrm{ac}}\nolimits}

\title{Ricci curvature, entropy and optimal transport}
\author{Shin-ichi OHTA\thanks{
Partly supported by the Grant-in-Aid for Young Scientists (B) 20740036.}\\
{\normalsize Department of Mathematics, Faculty of Science, Kyoto University,}\\
{\normalsize  Kyoto 606-8502, JAPAN (e-mail: {\sf sohta@math.kyoto-u.ac.jp})}}
\date{}
\begin{document}

\maketitle

\begin{abstract}
This is the lecture notes on the interplay between optimal transport
and Riemannian geometry.
On a Riemannian manifold, the convexity of entropy along optimal transport
in the space of probability measures characterizes lower bounds
of the Ricci curvature.
We then discuss geometric properties of general metric measure spaces
satisfying this convexity condition.
\medskip

\noindent {\bf Mathematics Subject Classification (2000)}: 53C21, 53C23, 53C60,
28A33, 28D20
\smallskip

\noindent {\bf Keywords}: Ricci curvature, entropy, optimal transport,
curvature-dimension condition
\end{abstract}

\section{Introduction}

This article is extended notes based on the author's lecture series
in summer school at Universit\'e Joseph Fourier, Grenoble:
`Optimal Transportation: Theory and Applications'.
The aim of these five lectures (corresponding to Sections~\ref{sc:intro}--\ref{sc:Fin})
was to review the recent impressive development
on the interplay between optimal transport theory and Riemannian geometry.
Ricci curvature and entropy are the key ingredients.
See \cite{Losig} for a survey in the same spirit with a slightly different
selection of topics.

Optimal transport theory is concerned with the behavior of transport
between two probability measures in a metric space.
We say that such transport is optimal if it minimizes a certain cost function
typically defined from the distance of the metric space.
Optimal transport naturally inherits the geometric structure of the
underlying space, especially Ricci curvature plays a crucial role for
describing optimal transport in Riemannian manifolds.
In fact, optimal transport is always performed along geodesics,
and we obtain Jacobi fields as their variational vector fields.
The behavior of these Jacobi fields is controlled
by the Ricci curvature as is usual in comparison geometry.
In this way, a lower Ricci curvature bound turns out to be equivalent
to a certain convexity property of entropy in the space of probability measures.
The latter convexity condition is called the curvature-dimension condition,
and it can be formulated without using the differentiable structure.
Therefore the curvature-dimension condition can be regarded as a
`definition' of a lower Ricci curvature bound for general metric measure spaces,
and implies many analogous properties in an interesting way.

A prerequisite is the basic knowledge of optimal transport theory
and Wasserstein geometry.
Riemannian geometry is also necessary in Sections~\ref{sc:intro}, \ref{sc:CD},
and is helpful for better understanding of the other sections.
We refer to \cite{AGS}, \cite{Vi1}, \cite{Vi2} and other articles in this proceeding
for optimal transport theory, \cite{CE}, \cite{Ch} and \cite{Sak}
for the basics of (comparison) Riemannian geometry.
We discuss Finsler geometry in Section~\ref{sc:Fin}, for which we refer to
\cite{BCS}, \cite{Shlec} and \cite{Oint}.
Besides them, main references are \cite{CMS1}, \cite{CMS2}, \cite{vRS},
\cite{StI}, \cite{StII}, \cite{LVI}, \cite{LVII} and \cite[Chapter~III]{Vi2}.

The organization of this article is as follows.
After summarizing some notations we use, Section~\ref{sc:intro} is devoted to
the definition of the Ricci curvature of Riemannian manifolds
and to the classical Bishop-Gromov volume comparison theorem.
In Section~\ref{sc:CD}, we start with the Brunn-Minkowski inequalities
in (unweighted or weighted) Euclidean spaces, and explain the equivalence between
a lower (weighted) Ricci curvature bound for a (weighted) Riemannian manifold
and the curvature-dimension condition.
In Section~\ref{sc:stab}, we give the precise definition of the curvature-dimension
condition for metric measure spaces, and see that it is stable under the measured
Gromov-Hausdorff convergence.
Section~\ref{sc:appl} is concerned with several geometric applications of the
curvature-dimension condition followed by related open questions.
In Section~\ref{sc:Fin}, we verify that this kind of machinery is useful
also in Finsler geometry.
We finally discuss three related topics in Section~\ref{sc:rel}.
Interested readers can find more references in Further Reading at the end of each section
(except the last section).

Some subjects in this article are more comprehensively discussed in \cite[Part III]{Vi2}.
Despite these inevitable overlaps with Villani's massive book,
we try to argue in a more geometric way, and mention recent development.
Analytic applications of the curvature-dimension condition are not dealt with
in these notes, for which we refer to \cite{LVII}, \cite{LVI} and
\cite[Chapter~III]{Vi2} among others.

I would like to express my gratitude to the organizers for the kind invitation to
the fascinating summer school, and to all the audience for their attendance and interest.
I also thank the referee for careful reading and valuable suggestions.

\section{Notations}\label{sc:not}

Throughout the article except Section~\ref{sc:Fin}, $(M,g)$ is an $n$-dimensional,
connected, complete $C^{\infty}$-Riemannian manifold without boundary
such that $n \ge 2$, $\vol_g$ stands for the Riemannian volume measure of $g$.
A {\it weighted Riemannian manifold} $(M,g,m)$ will mean a Riemannian manifold
$(M,g)$ endowed with a conformal deformation $m=e^{-\psi}\vol_g$
of $\vol_g$ with $\psi \in C^{\infty}(M)$.
Similarly, a weighted Euclidean space $(\R^n,\|\cdot\|,m)$ will be
a Euclidean space with a measure $m=e^{-\psi}\vol_n$, where $\vol_n$
stands for the $n$-dimensional Lebesgue measure.

A metric space is called a {\it geodesic space} if any two points $x,y \in X$ can be
connected by a rectifiable curve $\gamma:[0,1] \lra X$ of length $d(x,y)$
with $\gamma(0)=x$ and $\gamma(1)=y$.
Such minimizing curves parametrized proportionally to arc length are called
{\it minimal geodesics}.
The open ball of center $x$ and radius $r$ will be denoted by $B(x,r)$.
We remark that, thanks to the Hopf-Rinow theorem (cf.\ \cite[Theorem~2.4]{Bal}),
a complete, locally compact geodesic space is proper, i.e.,
every bounded closed set is compact.

In this article, we mean by a {\it metric measure space} a triple $(X,d,m)$
consisting of a complete, separable geodesic space $(X,d)$
and a Borel measure $m$ on it.
Our definition of the curvature-dimension condition will include
the additional (but natural) condition that $0<m(B(x,r))<\infty$ holds
for all $x \in X$ and $0<r<\infty$.
We extend $m$ to an outer measure in the Brunn-Minkowski inequalities
(Theorems~\ref{th:BMn}, \ref{th:BMN}, \ref{th:gBM},
see Remark~\ref{rm:mble} for more details).

For a complete, separable metric space $(X,d)$,
$\cP(X)$ stands for the set of Borel probability measures on $X$.
Define $\cP_2(X) \subset \cP(X)$ as the set of measures of finite second moment
(i.e., $\int_X d(x,y)^2 \,d\mu(y)<\infty$ for some (and hence all) $x \in X$).
We denote by $\cP_b(X) \subset \cP_2(X)$, $\cP_c(X) \subset \cP_b(X)$
the sets of measures of bounded or compact support, respectively.
Given a measure $m$ on $X$, denote by $\cP^{\ac}(X,m) \subset \cP(X)$
the set of absolutely continuous measures with respect to $m$.
Then $d^W_2$  stands for the
$L^2$-{\it $($Kantorovich-Rubinstein-$)$Wasserstein distance} of $\cP_2(X)$.
The push-forward of a measure $\mu$ by a map $\cF$ will be written as
$\cF_{\sharp}\mu$.

As usual in comparison geometry, the following functions will frequently appear in our discussions.
For $K \in \R$, $N \in (1,\infty)$ and $0<r$ ($<\pi\sqrt{(N-1)/K}$ if $K>0$), we set
\begin{equation}\label{eq:bs}
\bs_{K,N}(r) := \left\{
 \begin{array}{cl}
 \sqrt{(N-1)/K} \sin (r\sqrt{K/(N-1)}) & {\rm if}\ K>0, \vspace{1mm}\\
 r & {\rm if}\ K=0, \vspace{1mm}\\
 \sqrt{-(N-1)/K} \sinh (r\sqrt{-K/(N-1)}) & {\rm if}\ K<0.
 \end{array} \right.
\end{equation}
This is the solution to the differential equation
\begin{equation}\label{eq:bseq}
\bs_{K,N}''+\frac{K}{N-1}\bs_{K,N}=0
\end{equation}
with the initial conditions $\bs_{K,N}(0)=0$ and $\bs'_{K,N}(0)=1$.
For $n \in \N$ with $n \ge 2$, $\bs_{K,n}(r)^{n-1}$ is proportional to
the area of the sphere of radius $r$ in the $n$-dimensional space form
of constant sectional curvature $K/(n-1)$
(see Theorem~\ref{th:BG} and the paragraph after it).
In addition, using $\bs_{K,N}$, we define
\begin{equation}\label{eq:beta}
\beta^t_{K,N}(r) :=\bigg( \frac{\bs_{K,N}(tr)}{t\bs_{K,N}(r)} \bigg)^{N-1}, \qquad
 \beta^t_{K,\infty}(r) :=e^{K(1-t^2)r^2/6}
\end{equation}
for $K,N,r$ as above and $t \in (0,1)$.
This function plays a vital role in the key infinitesimal inequality $(\ref{eq:J})$ of
the curvature-dimension condition.

\section{Ricci curvature and comparison theorems}\label{sc:intro}

We begin with the basic concepts of curvature in Riemannian geometry and
several comparison theorems involving lower bounds of the Ricci curvature.
Instead of giving the detailed definition, we intend to explain
the geometric intuition of the sectional and Ricci curvatures
through comparison geometry.

Curvature is one of the most important quantities in Riemannian geometry.
By putting some conditions on the value of the (sectional or Ricci) curvature,
we obtain various quantitative and qualitative controls of distance,
measure, geodesics and so forth.
Comparison geometry is specifically interested in spaces whose curvature
is bounded by a constant from above or below.
In other words, we consider a space which is more positively or negatively
curved than a space form of constant curvature, and compare these spaces
from various viewpoints.

The $n$-dimensional (simply connected) {\it space form} $\bbM^n(k)$ of constant
sectional curvature $k \in \R$ is the unit sphere $\Sph^n$ for $k=1$;
the Euclidean space $\R^n$ for $k=0$; and the hyperbolic space $\Hy^n$ for $k=-1$.
Scaling gives general space forms for all $k \in \R$, e.g., $\bbM^n(k)$ for $k>0$
is the sphere of radius $1/\sqrt{k}$ in $\R^{n+1}$ with the induced
Riemannian metric.

\subsection{Sectional curvature}

Given linearly independent tangent vectors $v,w \in T_xM$,
the {\it sectional curvature} $\cK(v,w) \in \R$ reflects the asymptotic behavior
of the distance function $d(\gamma(t),\eta(t))$ near $t=0$ between geodesics
$\gamma(t)=\exp_x(tv)$ and $\eta(t)=\exp_x(tw)$.
That is to say, the asymptotic behavior of $d(t):=d(\gamma(t),\eta(t))$
near $t=0$ is same as the distance between geodesics,
with the same speed and angle between them,
in the space form of curvature $k=\cK(v,w)$.
(See Figure~1 which represents isometric embeddings of $\gamma$ and $\eta$
into $\R^2$ such that $d(\gamma(t),\eta(t))$ coincides with the Euclidean distance.)

\begin{center}
\begin{picture}(400,200)
\put(180,10){Figure~1}
\thicklines
\qbezier(70,70)(30,130)(30,190)
\qbezier(70,70)(110,130)(110,190)
\put(55,35){$\cK>0$}
\put(29,120){$\gamma$}
\put(105,120){$\eta$}
\put(66,58){$x$}

\put(200,70){\line(-1,3){40}}
\put(200,70){\line(1,3){40}}
\put(185,35){$\cK=0$}
\put(169,120){$\gamma$}
\put(225,120){$\eta$}
\put(196,58){$x$}

\qbezier(330,70)(325,140)(290,190)
\qbezier(330,70)(335,140)(370,190)
\put(315,35){$\cK<0$}
\put(308,120){$\gamma$}
\put(345,120){$\eta$}
\put(326,58){$x$}
\end{picture}
\end{center}

Assuming $\|v\|=\|w\|=1$ for simplicity, we can compute $d(t)$
in the space form $\bbM^n(k)$ by using the spherical/Euclidean/hyperbolic
law of cosines as (cf.\ \cite[Section~IV.1]{Sak})
\[ \begin{array}{rll}
\cos\big( \sqrt{k}d(t) \big)
&=\cos^2(\sqrt{k}t) +\sin^2(\sqrt{k}t)\cos\angle(v,w) & {\rm for}\ k>0, \\
d(t)^2
&= 2t^2-2t^2 \cos\angle(v,w) & {\rm for}\ k=0, \\
\cosh\big( \sqrt{-k}d(t) \big)
&=\cosh^2(\sqrt{-k}t) -\sinh^2(\sqrt{-k}t)\cos\angle(v,w) & {\rm for}\ k<0.
\end{array} \]
Observe that the dimension $n$ does not appear in these formulas.
The sectional curvature $\cK(v,w)$ depends only on the $2$-plane (in $T_xM$)
spanned by $v$ and $w$, and coincides with the Gaussian curvature at $x$ if $n=2$.

More precise relation between curvature and geodesics can be described
through Jacobi fields.
A $C^{\infty}$-vector field $J$ along a geodesic $\gamma:[0,l] \lra M$
is called a {\it Jacobi field} if it solves the {\it Jacobi equation}
\begin{equation}\label{eq:Ja}
D_{\dot{\gamma}}D_{\dot{\gamma}}J(t)
 +R\big( J(t),\dot{\gamma}(t) \big)\dot{\gamma}(t)=0
\end{equation}
for all $t \in [0,l]$.
Here $D_{\dot{\gamma}}$ denotes the covariant derivative along $\gamma$,
and $R:T_xM \otimes T_xM \lra T^*_xM \otimes T_xM$
is the {\it curvature tensor} determined by the Riemannian metric $g$.
Another equivalent way of introducing a Jacobi field is to define it as the
variational vector field $J(t)=(\del\sigma/\del s)(0,t)$ of some
$C^{\infty}$-variation $\sigma:(-\ve,\ve) \times [0,l] \lra M$ such that
$\sigma(0,t)=\gamma(t)$ and that every $\sigma_s:=\sigma(s,\cdot)$ is geodesic.
(This characterization of Jacobi fields needs only the class of geodesics,
and then it is possible to regard $(\ref{eq:Ja})$ as the definition of $R$.)
For linearly independent vectors $v,w \in T_xM$,
the precise definition of the sectional curvature is
\[ \cK(v,w):=\frac{\langle R(w,v)v,w \rangle}{\|v\|^2 \|w\|^2-\langle v,w \rangle^2}. \]
It might be helpful to compare $(\ref{eq:Ja})$ with $(\ref{eq:bseq})$.

\begin{remark}[Alexandrov spaces]\label{rm:Alex}
Although it is not our main subject, we briefly comment on
comparison geometry involving lower bounds of the sectional curvature.
As the sectional curvature is defined for each two-dimensional subspace
in tangent spaces, it controls the behavior of two-dimensional subsets in $M$,
in particular, triangles.
The classical {\it Alexandrov-Toponogov comparison theorem} asserts that
$\cK \ge k$ holds for some $k \in \R$ if and only if every geodesic triangle in $M$
is thicker than the triangle with the same side lengths in $\bbM^2(k)$.
See Figure~2 for more details, where $M$ is of $\cK \ge k$,
and then $d(x,w) \ge d(\tilde{x},\tilde{w})$ holds between geodesic triangles
with the same side lengths ($d(x,y)=d(\tilde{x},\tilde{y})$,
$d(y,z)=d(\tilde{y},\tilde{z})$, $d(z,x)=d(\tilde{z},\tilde{x})$)
as well as $d(y,w)=d(\tilde{y},\tilde{w})$.

\begin{center}
\begin{picture}(400,180)
\put(180,10){Figure~2}
\put(160,40){$d(x,w) \ge d(\tilde{x},\tilde{w})$}

\qbezier(100,50)(60,90)(50,150)
\qbezier(100,50)(140,90)(150,150)
\qbezier(50,150)(100,160)(150,150)
\put(95,35){$x$}
\put(45,158){$y$}
\put(70,162){$w$}
\put(145,158){$z$}
\put(20,100){$M \supset$}

\put(300,50){\line(1,2){50}}
\put(300,50){\line(-1,2){50}}
\put(250,150){\line(1,0){100}}
\put(297,35){$\tilde{x}$}
\put(245,158){$\tilde{y}$}
\put(345,158){$\tilde{z}$}
\put(270,158){$\tilde{w}$}
\put(350,100){$\subset \bbM^2(k)$}

\thicklines
\put(74.2,153.3){\line(1,-4){25.7}}
\put(275,150){\line(1,-4){25}}
\end{picture}
\end{center}

The point is that we can forget about the dimension of $M$,
because the sectional curvature cares only two-dimensional subsets.
The above triangle comparison property is written by using only distance
and geodesics, so that it can be formulated for metric spaces having
enough geodesics (i.e., geodesic spaces).
Such spaces are called {\it Alexandrov spaces}, and there are deep
geometric and analytic theories on them
(see \cite{BGP}, \cite{OtS}, \cite[Chapters~4, 10]{BBI}).
We discuss optimal transport and Wasserstein geometry
on Alexandrov spaces in Subsection~\ref{ssc:Al}.
\end{remark}

\subsection{Ricci curvature}

Given a unit vector $v \in T_xM$, we define the {\it Ricci curvature} of $v$
as the trace of the sectional curvature $\cK(v,\cdot)$,
\[ \Ric(v):=\sum_{i=1}^{n-1} \cK(v,e_i), \]
where $\{ e_i \}_{i=1}^{n-1} \cup \{v\}$ is an orthonormal basis of $T_xM$.
We will mean by $\Ric \ge K$ for $K \in \R$ that $\Ric(v) \ge K$ holds
for all unit vectors $v \in TM$.
As we discussed in the previous subsection,
sectional curvature controls geodesics and distance.
Ricci curvature has less information since we take the trace,
and naturally controls the behavior of the measure $\vol_g$.

The following is one of the most important theorems
in comparison Riemannian geometry, that asserts that a lower bound
of the Ricci curvature implies an upper bound of the volume growth.
The proof is done via calculations involving Jacobi fields.
Recall $(\ref{eq:bs})$ for the definition of the function $\bs_{K,n}$.

\begin{theorem}[Bishop-Gromov volume comparison]\label{th:BG}
Assume that $\Ric \ge K$ holds for some $K \in \R$.
Then we have, for any $x \in M$ and $0<r<R\ (\le \pi\sqrt{(n-1)/K}$ if $K>0)$,
\begin{equation}\label{eq:BG}
\frac{\vol_g(B(x,R))}{\vol_g(B(x,r))} \le
 \frac{\int_0^R \bs_{K,n}(t)^{n-1}\,dt}{\int_0^r \bs_{K,n}(t)^{n-1}\,dt}.
\end{equation}
\end{theorem}

\begin{proof}
Given a unit vector $v \in T_xM$, we fix a unit speed minimal geodesic
$\gamma:[0,l] \lra M$ with $\dot{\gamma}(0)=v$ and 
an orthonormal basis $\{ e_i \}_{i=1}^{n-1} \cup \{v\}$ of $T_xM$.
Then we consider the variation $\sigma_i:(-\ve,\ve) \times [0,l] \lra M$ defined by
$\sigma_i(s,t):=\exp_x(tv+ste_i)$ for $i=1,\ldots,n-1$,
and introduce the Jacobi fields $\{ J_i \}_{i=1}^{n-1}$ along $\gamma$ given by
\[ J_i(t):=\frac{\del \sigma_i}{\del s}(0,t)
 =D(\exp_x)_{tv}(te_i) \in T_{\gamma(t)}M \]
(see Figure~3, where $s>0$).

\begin{center}
\begin{picture}(400,120)
\put(180,10){Figure~3}

\thicklines
\qbezier(100,50)(200,60)(300,50)
\qbezier(100,50)(200,90)(290,80)

\put(140,53){\vector(0,1){12}}
\put(180,55){\vector(0,1){24}}
\put(220,55){\vector(0,1){36}}
\put(260,53){\vector(0,1){48}}

\put(85,45){$x$}
\put(310,45){$\gamma$}
\put(300,75){$\sigma_i(s,\cdot)$}
\put(245,98){$J_i$}

\end{picture}
\end{center}

Note that $J_i(0)=0$, $D_{\dot{\gamma}}J_i(0)=e_i$,
$\langle J_i,\dot{\gamma} \rangle \equiv 0$ (by the Gauss lemma) and
$\langle D_{\dot{\gamma}}J_i,\dot{\gamma} \rangle \equiv 0$
(by $(\ref{eq:Ja})$ and
$\langle R(J_i,\dot{\gamma})\dot{\gamma},\dot{\gamma} \rangle \equiv 0$).
We also remark that $\gamma(t)$ is not conjugate to $x$ for all $t \in (0,l)$
(and hence $\{J_i(t)\}_{i=1}^{n-1} \cup \{\dot{\gamma}(t)\}$ is a basis of
$T_{\gamma(t)}M$) since $\gamma$ is minimal.
Hence we find an $(n-1) \times (n-1)$ matrix $\cU(t)=(u_{ij}(t))_{i,j=1}^{n-1}$
such that $D_{\dot{\gamma}}J_i(t)=\sum_{j=1}^{n-1}u_{ij}(t)J_j(t)$ for $t \in (0,l)$.
We define two more $(n-1) \times (n-1)$ matrices
\[ \cA(t) :=\big( \langle J_i(t),J_j(t) \rangle \big)_{i,j=1}^{n-1}, \quad
 \cR(t) :=\Big( \big\langle R\big( J_i(t),\dot{\gamma}(t)\big) \dot{\gamma}(t),
 J_j(t) \big\rangle \Big)_{i,j=1}^{n-1}. \]
Note that $\cA$ and $\cR$ are symmetric matrices.
Moreover, we have $\tr(\cR(t)\cA(t)^{-1})=\Ric(\dot{\gamma}(t))$
as $\cA(t)$ is the matrix representation of the metric $g$
in the basis $\{J_i(t)\}_{i=1}^{n-1}$
of the orthogonal complement $\dot{\gamma}(t)^{\perp}$ of $\dot{\gamma}(t)$.
To be precise, choosing an $(n-1) \times (n-1)$ matrix $\cC=(c_{ij})_{i,j=1}^{n-1}$
such that $\{ \sum_{j=1}^{n-1}c_{ij}J_j(t) \}_{i=1}^{n-1}$ is orthonormal,
we observe $I_n=\cC \cA \cC^t$ ($\cC^t$ is the transpose of $\cC$) and
\begin{align*}
\Ric\big( \dot{\gamma}(t) \big)
&=\sum_{i,j,k=1}^{n-1} \big\langle R\big( c_{ij}J_j(t),\dot{\gamma}(t) \big)
 \dot{\gamma}(t),c_{ik}J_k(t) \big\rangle
 =\tr\big( \cC(t) \cR(t) \cC(t)^t \big) \\
&=\tr\big( \cR(t) \cA(t)^{-1} \big).
\end{align*}

\begin{claim}\label{cl:BG}
\begin{enumerate}[{\rm (a)}]
\item It holds that $\cU\cA=\cA\cU^t$.
In particular, we have $2\cU=\cA'\cA^{-1}$.
\item The matrix $\cU$ is symmetric and we have $\tr(\cU^2) \ge (\tr \cU)^2/(n-1)$.
\end{enumerate}
\end{claim}

\begin{proof}
(a) The first assertion easily follows from the Jacobi equation $(\ref{eq:Ja})$
and the symmetry of $\cR$, indeed,
\begin{align*}
\frac{d}{dt}\{ \langle D_{\dot{\gamma}}J_i,J_j \rangle
 -\langle J_i,D_{\dot{\gamma}}J_j \rangle \}
&= \langle D_{\dot{\gamma}} D_{\dot{\gamma}}J_i,J_j \rangle
 -\langle J_i,D_{\dot{\gamma}} D_{\dot{\gamma}}J_j \rangle \\
&= -\langle R(J_i,\dot{\gamma})\dot{\gamma},J_j \rangle
 +\langle J_i,R(J_j,\dot{\gamma})\dot{\gamma} \rangle
=0.
\end{align*}
Thus we have $\cA'=\cU\cA+\cA\cU^t=2\cU\cA$ which shows the second assertion.

(b) Recall that
\[ \frac{\del \sigma_i}{\del t}(0,t)=\dot{\gamma}(t), \qquad
 \frac{\del \sigma_i}{\del s}(0,t)=J_i(t) \]
hold for $t \in (0,l)$.
As $[\del/\del s,\del/\del t]=0$, we have
\[ D_{\dot{\gamma}}J_i(t)=D_t\bigg( \frac{\del \sigma_i}{\del s} \bigg)(0,t)
 =D_s\bigg( \frac{\del \sigma_i}{\del t} \bigg)(0,t). \]
Now, we introduce the function
\[ f:\exp_x \bigg( \bigg\{ tv+\sum_{i=1}^{n-1}s_i t e_i \,\Big|\,
 t \in [0,l],\ |s_i|<\ve \bigg\} \bigg) \lra \R \]
so that $f(\exp_x(tv+\sum_{i=1}^{n-1} s_i t e_i))=t$.
We derive from $\nabla f(\sigma_i(s,t))=(\del\sigma_i/\del t)(s,t)$ that
\[ D_s\bigg( \frac{\del \sigma_i}{\del t} \bigg)(0,t)
 =D_{J_i}(\nabla f)\big( \gamma(t) \big) =\nabla^2 f\big( J_i(t) \big), \]
where $\langle \nabla^2 f(w),w' \rangle=\Hess f(w,w')$.
This means that $\cU$ is the matrix presentation of the symmetric form
$\nabla^2 f$ (restricted in $\dot{\gamma}^{\perp}$) with respect to
the basis $\{ J_i \}_{i=1}^{n-1}$.
Therefore $\cU$ is symmetric.
By denoting the eigenvalues of $\cU$ by $\lambda_1,\ldots,\lambda_{n-1}$,
the Cauchy-Schwarz inequality shows that
\[ (\tr \cU)^2 =\bigg( \sum_{i=1}^{n-1}\lambda_i \bigg)^2
 \le (n-1) \sum_{i=1}^{n-1} \lambda_i^2 =(n-1)\tr(\cU^2). \]
$\hfill \diamondsuit$
\end{proof}

We calculate, by using Claim~\ref{cl:BG}(a),
\begin{align*}
\big[ (\det\cA)^{1/2(n-1)} \big]'
&= \frac{1}{2(n-1)}(\det\cA)^{1/2(n-1)-1} \cdot \det\cA \tr(\cA'\cA^{-1}) \\
&= \frac{1}{n-1}(\det\cA)^{1/2(n-1)} \tr\cU.
\end{align*}
Then Claim~\ref{cl:BG}(b) yields
\begin{align*}
\big[ (\det\cA)^{1/2(n-1)} \big]''
&= \frac{1}{(n-1)^2}(\det\cA)^{1/2(n-1)}(\tr\cU)^2
 +\frac{1}{n-1}(\det\cA)^{1/2(n-1)} \tr(\cU') \\
&\le \frac{1}{n-1}(\det\cA)^{1/2(n-1)}\{ \tr(\cU^2)+\tr(\cU') \}.
\end{align*}
We also deduce from Claim~\ref{cl:BG}(a) and $(\ref{eq:Ja})$ that
\[ \cU' =\frac{1}{2}\cA''\cA^{-1} -\frac{1}{2}(\cA'\cA^{-1})^2
 =\frac{1}{2}(-2\cR +2\cU\cA\cU)\cA^{-1} -2\cU^2
 =-\cR\cA^{-1}-\cU^2. \]
This implies the (matrix) {\it Riccati equation}
\[ \cU'+\cU^2+\cR \cA^{-1} =0. \]
Taking the trace gives
\[ (\tr \cU)' +\tr (\cU^2) +\Ric(\dot{\gamma})=0. \]
Thus we obtain from our hypothesis $\Ric \ge K$ the differential inequality
\begin{equation}\label{eq:Bish}
\big[ (\det\cA)^{1/2(n-1)} \big]'' \le -\frac{K}{n-1} (\det\cA)^{1/2(n-1)}.
\end{equation}
This is a version of the fundamental {\it Bishop comparison theorem}
which plays a prominent role in comparison geometry.
Comparing $(\ref{eq:Bish})$ with $(\ref{eq:bseq})$, we have
\begin{align*}
&\frac{d}{dt}\Big\{ \big[ (\det \cA)^{1/2(n-1)} \big]' \bs_{K,n}
 -(\det \cA)^{1/2(n-1)} \bs'_{K,n} \Big\} \\
&= \big[ (\det \cA)^{1/2(n-1)} \big]'' \bs_{K,n}
 -(\det \cA)^{1/2(n-1)} \bs''_{K,n}
 \le 0,
\end{align*}
and hence $(\det \cA)^{1/2(n-1)}/\bs_{K,n}$ is non-increasing.
Then integrating $\sqrt{\det \cA}$ in unit vectors $v \in T_xM$ implies
the {\it area comparison theorem}
\begin{equation}\label{eq:Bish2}
\frac{\area_g(S(x,R))}{\area_g(S(x,r))}
 \le \frac{\bs_{K,n}(R)^{n-1}}{\bs_{K,n}(r)^{n-1}},
\end{equation}
where $S(x,r):=\{ y \in M \,|\, d(x,y)=r \}$ and $\area_g$ stands for
the $(n-1)$-dimensional Hausdorff measure associated with $g$
(in other words, the volume measure of the $(n-1)$-dimensional Riemannian
metric of $S(x,r)$ induced from $g$).

Now, we integrate $(\ref{eq:Bish2})$ in the radial direction.
Set $\bA(t):=\area_g(S(x,t))$ and $\bS(t):=\bs_{K,n}(t)^{n-1}$,
and recall that $\bA/\bS$ is non-increasing.
Hence we obtain the key inequality
\begin{equation}\label{eq:BG-}
\int_0^r \bA \,dt \int_r^R \bS \,dt
 \ge \frac{\bA(r)}{\bS(r)} \int_0^r \bS \,dt \int_r^R \bS \,dt
 \ge \int_0^r \bS \,dt \int_r^R \bA \,dt.
\end{equation}
From here to the desired estimate $(\ref{eq:BG})$ is the easy calculation
as follows
\begin{align*}
&\vol_g\big( B(x,r) \big) \int_0^R \bs_{K,n}(t)^{n-1} \,dt
 =\int_0^r \bA \,dt \int_r^R \bS \,dt +\int_0^r \bA \,dt \int_0^r \bS \,dt \\
&\ge \int_0^r \bS \,dt \int_r^R \bA \,dt +\int_0^r \bA \,dt \int_0^r \bS \,dt
 = \vol_g\big( B(x,R) \big) \int_0^r \bs_{K,n}(t)^{n-1} \,dt.
\end{align*}
$\qedd$
\end{proof}

The sphere of radius $r$ in the space form $\bbM^n(k)$ has area
$a_n\bs_{(n-1)k,n}(r)^{n-1}$, where $a_n$ is the area of $\Sph^{n-1}$,
and the ball of radius $r$ has volume
$a_n\int_0^r \bs_{(n-1)k,n}(t)^{n-1} \,dt$.
Thus the right-hand side of $(\ref{eq:BG})$ ($(\ref{eq:Bish2})$, respectively)
coincides with the ratio of the volume of balls (the area of spheres, respectively)
of radius $R$ and $r$ in $\bbM^n(K/(n-1))$.

Theorem~\ref{th:BG} for $K>0$ immediately implies a diameter bound.
This ensures that the condition $R \le \pi\sqrt{(n-1)/K}$ in Theorem~\ref{th:BG}
is natural.

\begin{corollary}[Bonnet-Myers diameter bound]\label{cr:BG}
If $\Ric \ge K>0$, then we have
\begin{equation}\label{eq:BMy}
\diam M \le \pi\sqrt{\frac{n-1}{K}}.
\end{equation}
\end{corollary}

\begin{proof}
Put $R:=\pi\sqrt{(n-1)/K}$ and assume $\diam M \ge R$.
Given $x \in M$, Theorem~\ref{th:BG} implies that
\begin{align*}
\limsup_{\ve \downarrow 0}
 \frac{\vol_g(B(x,R) \setminus B(x,R-\ve))}{\vol_g(B(x,R))}
&=\limsup_{\ve \downarrow 0}
 \bigg\{ 1-\frac{\vol_g(B(x,R-\ve))}{\vol_g(B(x,R))} \bigg\} \\
&\le \limsup_{\ve \downarrow 0}
 \frac{\int_{R-\ve}^R \bs_{K,n}(t)^{n-1}\,dt}{\int_0^R \bs_{K,n}(t)^{n-1}\,dt}
 =0.
\end{align*}
This shows $\area_g(S(x,R))=0$ and hence $\diam M \le R$.
To be precise, it follows from $\area_g(S(x,R))=0$ that every point in $S(x,R)$
must be a conjugate point of $x$.
Therefore any geodesic emanating from $x$ is not minimal after
passing through $S(x,R)$, and hence $\diam M=R$.
(A more direct proof in terms of metric geometry can be found
in Theorem~\ref{th:BMy}(i).)
$\qedd$
\end{proof}

The bound $(\ref{eq:BMy})$ is sharp, and equality is achieved only by
the sphere in $\bbM^n(K/(n-1))$ of radius $\sqrt{(n-1)/K}$
(compare this with Theorem~\ref{th:sph}).

As we mentioned in Remark~\ref{rm:Alex}, lower sectional curvature bounds
are characterized by simple triangle comparison properties involving only distance,
and there is a successful theory of metric spaces satisfying them.
Then it is natural to ask the following question.

\begin{question}\label{Q:Ric}
How to characterize lower Ricci curvature bounds without using differentiable
structure?
\end{question}

This had been a long standing important question, and we will see
an answer in the next section (Theorem~\ref{th:CD}).
Such a condition naturally involves measure and dimension besides distance,
and should be preserved under the convergence of metric measure spaces
(see Section~\ref{sc:stab}).

\begin{fread}
See, for instances, \cite{CE}, \cite{Ch} and \cite{Sak}
for the fundamentals of Riemannian geometry and comparison theorems.
A property corresponding to the Bishop comparison theorem $(\ref{eq:Bish})$
was proposed as a lower Ricci curvature bound for metric measure spaces
by Cheeger and Colding~\cite{CC} (as well as Gromov~\cite{Gr}),
and used to study the limit spaces of Riemannian manifolds
with uniform lower Ricci curvature bounds.
The deep theory of such limit spaces is one of the main motivations for asking
Question~\ref{Q:Ric}, so that the stability deserves a particular interest
(see Section~\ref{sc:stab} for more details).
The systematic investigation of $(\ref{eq:Bish})$ in metric measure spaces
has not been done until \cite{Omcp} and \cite{StII}
where we call this property the measure contraction property.
We will revisit this in Subsection~\ref{ssc:mcp}.
Here we only remark that the measure contraction property is
strictly weaker than the curvature-dimension
condition.
\end{fread}

\section{A characterization of lower Ricci curvature bound via optimal transport}
\label{sc:CD}

The Bishop-Gromov volume comparison theorem (Theorem~\ref{th:BG})
can be regarded as a concavity estimate of $\vol_g^{1/n}$
along the contraction of the ball $B(x,R)$ to its center $x$.
This is generalized to optimal transport between pairs of uniform distributions
(the Brunn-Minkowski inequality) and, moreover,
pairs of probability measures (the curvature-dimension condition).
Figure~4 represents the difference between contraction
and transport (see also Figures~5, 8).

\begin{center}
\begin{picture}(400,180)
\put(180,10){Figure~4}
\put(40,30){Bishop-Gromov}
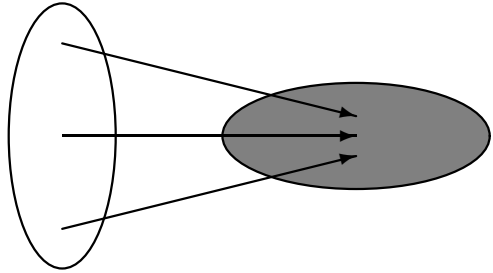
\put(180,30){Brunn-Minkowski/Curvature-dimension}

\thicklines
\put(80,110){\circle{100}}
\put(80,110){\shade\circle{60}}
\put(80,150){\vector(0,-1){20}}
\put(80,70){\vector(0,1){20}}
\put(40,110){\vector(1,0){20}}
\put(120,110){\vector(-1,0){20}}
\put(79,109){\rule{2pt}{2pt}}

\put(210,110){\ellipse{40}{100}}
\put(320,110){\shade\ellipse{100}{40}}
\put(210,110){\vector(1,0){110}}
\put(210,75){\vector(4,1){110}}
\put(210,145){\vector(4,-1){110}}

\end{picture}
\end{center}

The main theorem in this section is Theorem~\ref{th:CD} which asserts that,
on a weighted Riemannian manifold,
the curvature-dimension condition is equivalent to
a lower bound of the corresponding weighted Ricci curvature.
In order to avoid lengthy calculations, we begin with Euclidean spaces
with or without weight, and see the relation between
the Brunn-Minkowski inequality and the weighted Ricci curvature.
Then the general Riemannian situation is only briefly explained.
We hope that our simplified argument will help the readers to catch
the idea of the curvature-dimension condition.

\subsection{Brunn-Minkowski inequalities in Euclidean spaces}\label{ssc:BM}

For later convenience, we explain fundamental facts of optimal transport theory
on Euclidean spaces.
Given $\mu_0,\mu_1 \in \cP_c(\R^n)$ with $\mu_0 \in \cP^{\ac}(\R^n,\vol_n)$,
there is a convex function $f:\R^n \lra \R$ such that the map
\[ \cF_t(x):=(1-t)x +t\nabla f(x), \quad t \in [0,1], \]
gives the unique optimal transport from $\mu_0$ to $\mu_1$
(Brenier's theorem, \cite{Br}).
Precisely, $t \longmapsto \mu_t:=(\cF_t)_{\sharp}\mu_0$ is the unique
minimal geodesic from $\mu_0$ to $\mu_1$ with respect to the
$L^2$-Wasserstein distance.
We remark that the convex function $f$ is twice differentiable a.e.\
(Alexandrov's theorem, cf.\ \cite[Chapter~14]{Vi2}).
Thus $\nabla f$ makes sense and $\cF_t$ is differentiable a.e.
Moreover, the {\it Monge-Amp\`ere equation}
\begin{equation}\label{eq:MA}
\rho_1\big( \cF_1(x) \big) \det\big( D\cF_1(x) \big) =\rho_0(x)
\end{equation}
holds for $\mu_0$-a.e.\ $x$.

Now, we are ready for proving the classical Brunn-Minkowski inequality
in the (unweighted) Euclidean space $(\R^n,\|\cdot\|,\vol_n)$.
Briefly speaking, it asserts that $\vol_n^{1/n}$ is concave.
We shall give a proof based on optimal transport theory.
Given two (nonempty) sets $A,B \subset \R^n$ and $t \in [0,1]$,
we set
\[ (1-t)A+tB:=\{ (1-t)x+ty \,|\, x \in A,\ y \in B \}. \]
(See Figure~5, where $(1/2)A+(1/2)B$ has much more measure than $A$ and $B$.)

\begin{center}
\begin{picture}(400,180)
\put(180,10){Figure~5}

\thicklines
\put(60,100){\shade\ellipse{20}{120}}
\put(320,100){\shade\ellipse{120}{20}}

\qbezier(165,67)(190,61)(215,67)
\qbezier(165,133)(190,139)(215,133)
\qbezier(223,125)(229,100)(223,75)
\qbezier(157,125)(151,100)(157,75)
\qbezier(165,67)(159,69)(157,75)
\qbezier(165,133)(159,131)(157,125)
\qbezier(215,67)(221,69)(223,75)
\qbezier(215,133)(221,131)(223,125)

\put(60,100){\line(1,0){260}}
\qbezier(60,40)(215,67)(370,94)
\qbezier(60,160)(215,133)(370,106)

\put(59,99){\rule{2pt}{2pt}}
\put(189,99){\rule{2pt}{2pt}}
\put(319,99){\rule{2pt}{2pt}}

\put(59,39){\rule{2pt}{2pt}}
\put(214,66){\rule{2pt}{2pt}}
\put(369,93){\rule{2pt}{2pt}}

\put(59,159){\rule{2pt}{2pt}}
\put(214,132){\rule{2pt}{2pt}}
\put(369,105){\rule{2pt}{2pt}}

\put(30,95){$A$}
\put(150,40){$(1/2)A+(1/2)B$}
\put(315,70){$B$}

\end{picture}
\end{center}

\begin{theorem}[Brunn-Minkowski inequality]\label{th:BMn}
For any measurable sets $A,B \subset \R^n$ and $t \in [0,1]$, we have
\begin{equation}\label{eq:BM}
\vol_n\big( (1-t)A+tB \big)^{1/n} \ge (1-t)\vol_n(A)^{1/n} +t\vol_n(B)^{1/n}.
\end{equation}
\end{theorem}

\begin{proof}
We can assume that both $A$ and $B$ are bounded and of positive measure.
The case of $\vol_n(A)=0$ is easily checked by choosing
a point $x \in A$, as we have
\[ \vol_n\big( (1-t)A+tB \big)^{1/n} \ge \vol_n\big( (1-t)\{x\}+tB \big)^{1/n}
 =t\vol_n(B)^{1/n}. \]
If either $A$ or $B$ is unbounded, then applying $(\ref{eq:BM})$ to bounded sets yields
\begin{align*}
&\vol_n\big( (1-t)\{ A \cap B(0,R)\}+t\{ B \cap B(0,R)\} \big)^{1/n} \\
&\ge (1-t)\vol_n\big( A \cap B(0,R) \big)^{1/n}
 +t\vol_n\big( B \cap B(0,R) \big)^{1/n}.
\end{align*}
We take the limit as $R$ go to infinity and obtain
\[  \vol_n\big( (1-t)A+tB \big)^{1/n} \ge
 (1-t)\vol_n(A)^{1/n} +t\vol_n(B)^{1/n}. \]

Consider the uniform distributions on $A$ and $B$,
\[ \mu_0=\rho_0 \vol_n:=\frac{\chi_A}{\vol_n(A)}\vol_n, \qquad
 \mu_1=\rho_1 \vol_n:=\frac{\chi_B}{\vol_n(B)}\vol_n, \]
where $\chi_A$ stands for the characteristic function of $A$.
As $\mu_0$ is absolutely continuous, there is a convex function
$f:\R^n \lra \R$ such that the map $\cF_1=(1-t)\Id_{\R^n}+t\nabla f$,
$t \in [0,1]$, is the unique optimal transport from $\mu_0$ to $\mu_1$.
Between the uniform distributions $\mu_0$ and $\mu_1$,
the Monge-Amp\`ere equation $(\ref{eq:MA})$ simply means that
\[ \det(D\cF_1) =\frac{\vol_n(B)}{\vol_n(A)} \]
$\mu_0$-a.e.

Note that $D\cF_1=\Hess f$ is symmetric and positive definite $\mu_0$-a.e.,
since $f$ is convex and $\det(D\cF_1)>0$.
We shall estimate $\det(D\cF_t)=\det((1-t)I_n+tD\cF_1)$ from above and below.
To do so, we denote the eigenvalues of $D\cF_1$ by
$\lambda_1,\ldots,\lambda_n>0$ and apply the inequality of
arithmetic and geometric means to see
\begin{align*}
&\bigg\{ \frac{(1-t)^n}{\det((1-t)I_n+tD\cF_1)} \bigg\}^{1/n}
 +\bigg\{ \frac{t^n \det(D\cF_1)}{\det((1-t)I_n+tD\cF_1)} \bigg\}^{1/n} \\
&= \bigg\{ \prod_{i=1}^n \frac{1-t}{(1-t)+t\lambda_i} \bigg\}^{1/n}
 +\bigg\{ \prod_{i=1}^n \frac{t\lambda_i}{(1-t)+t\lambda_i} \bigg\}^{1/n} \\
&\le \frac{1}{n} \sum_{i=1}^n \bigg\{ \frac{1-t}{(1-t)+t\lambda_i}
 +\frac{t\lambda_i}{(1-t)+t\lambda_i} \bigg\} =1.
\end{align*}
Thus we have, on the one hand,
\begin{equation}\label{eq:a-g}
\det(D\cF_t)^{1/n} \ge (1-t)+t\det(D\cF_1)^{1/n}
 =(1-t)+t\bigg\{ \frac{\vol_n(B)}{\vol_n(A)} \bigg\}^{1/n}.
\end{equation}
On the other hand, the H\"older inequality and the change of variables formula yield
\[ \int_A \det(D\cF_t)^{1/n} \,d\mu_0
 \le \bigg( \int_A \det(D\cF_t) \,d\mu_0 \bigg)^{1/n}
 =\bigg( \frac{1}{\vol_n(A)} \int_{\cF_t(A)} d\vol_n \bigg)^{1/n}. \]
Therefore we obtain
\[ \int_A \det(D\cF_t)^{1/n} \,d\mu_0
 \le \bigg\{ \frac{\vol_n(\cF_t(A))}{\vol_n(A)} \bigg\}^{1/n}
 \le \bigg\{ \frac{\vol_n((1-t)A+tB)}{\vol_n(A)} \bigg\}^{1/n}. \]
Combining these, we complete the proof of $(\ref{eq:BM})$.
$\qedd$
\end{proof}

\begin{remark}\label{rm:mble}
We remark that the set $(1-t)A+tB$ is not necessarily measurable
(regardless the measurability of $A$ and $B$).
Hence, to be precise, $\vol_n((1-t)A+tB)$ is considered as an outer measure
given by $\inf_W \vol_n(W)$, where $W \subset \R^n$ runs over all
measurable sets containing $(1-t)A+tB$.
The same remark is applied to Theorems~\ref{th:BMN}, \ref{th:gBM} below.
\end{remark}

Next we treat the weighted case $(\R^n,\|\cdot\|,m)$,
where $m=e^{-\psi}\vol_n$ with $\psi \in C^{\infty}(\R^n)$.
Then we need to replace $1/n$ in $(\ref{eq:BM})$ with $1/N$ for some
$N \in (n,\infty)$, and the analogue of $(\ref{eq:BM})$
leads us to an important condition on $\psi$.

\begin{theorem}[Brunn-Minkowski inequality with weight]\label{th:BMN}
Take $N \in (n,\infty)$.
A weighted Euclidean space $(\R^n,\|\cdot\|,m)$ with $m=e^{-\psi}\vol_n$,
$\psi \in C^{\infty}(\R^n)$, satisfies
\begin{equation}\label{eq:BMN}
m\big( (1-t)A+tB \big)^{1/N} \ge (1-t)m(A)^{1/N} +tm(B)^{1/N}
\end{equation}
for all measurable sets $A,B \subset \R^n$ and all $t \in [0,1]$
if and only if
\begin{equation}\label{eq:NR}
\Hess\psi(v,v)-\frac{\langle \nabla\psi(x),v \rangle^2}{N-n} \ge 0
\end{equation}
holds for all unit vectors $v \in T_x \R^n$.
\end{theorem}

\begin{proof}
We first prove that $(\ref{eq:NR})$ implies $(\ref{eq:BMN})$.
Similarly to Theorem~\ref{th:BMn}, we assume that $A$ and $B$
are bounded and of positive measure, and set
\[ \mu_0:=\frac{\chi_A}{m(A)}m, \qquad \mu_1:=\frac{\chi_B}{m(B)}m. \]
We again find a convex function $f:\R^n \lra \R$ such that
$\mu_t:=(\cF_t)_{\sharp}\mu_0$ is the minimal geodesic from
$\mu_0$ to $\mu_1$, where $\cF_t:=(1-t)\Id_{\R^n}+t\nabla f$.
Instead of $\det(D\cF_t)$, we consider
\[ \det_m\big( D\cF_t(x) \big) :=e^{\psi(x)-\psi(\cF_t(x))} \det\big( D\cF_t(x) \big). \]
The coefficient $e^{\psi(x)-\psi(\cF_t(x))}$ represents the ratio
of the weights at $x$ and $\cF_t(x)$.
As in Theorem~\ref{th:BMn} (see, especially, $(\ref{eq:a-g})$),
it is sufficient to show the concavity of
$\det_m(D\cF_t(x))^{1/N}$ to derive the desired inequality $(\ref{eq:BMN})$.
Fix $x \in A$ and put
\[ \gamma(t):=\cF_t(x), \quad \Phi_m(t):=\det_m\big( D\cF_t(x) \big)^{1/N},
 \quad \Phi(t):=\det\big( D\cF_t(x) \big)^{1/n}. \]
On the one hand, it is proved in $(\ref{eq:a-g})$ that
$\Phi(t) \ge (1-t)\Phi(0)+t\Phi(1)$.
On the other hand, the assumption $(\ref{eq:NR})$ implies that
$e^{-\psi(\cF_t(x))/(N-n)}$ is a concave function in $t$.
These together imply $(\ref{eq:BMN})$ via the H\"older inequality.
To be precise, we have
\begin{align*}
&\Phi_m(t) =e^{\{ \psi(x)-\psi(\cF_t(x)) \}/N} \Phi(t)^{n/N} \\
&\ge e^{\psi(x)/N}
 \big\{ (1-t)e^{-\psi(x)/(N-n)} +te^{-\psi(\cF_1(x))/(N-n)} \big\}^{(N-n)/N}
 \big\{ (1-t)\Phi(0)+t\Phi(1) \big\}^{n/N},
\end{align*}
and then the H\"older inequality yields
\begin{align*}
\Phi_m(t) &\ge e^{\psi(x)/N}
 \big\{ (1-t)e^{-\psi(x)/N} \Phi(0)^{n/N} +te^{-\psi(\cF_1(x))/N} \Phi(1)^{n/N} \big\} \\
&= (1-t)\Phi_m(0) +t\Phi_m(1).
\end{align*}

To see the converse, we fix an arbitrary unit vector $v \in T_x\R^n$
and set $\gamma(t):=x+tv$ for $t \in \R$ and
$a:=\langle \nabla\psi(x),v \rangle/(N-n)$.
Given $\ve>0$ and $\delta \in \R$ with $\ve, |\delta| \ll 1$,
we consider two open balls (see Figure~6 where $a\delta>0$)
\[ A_+:=B\big( \gamma(\delta),\ve(1-a\delta) \big), \qquad
 A_-:=B\big( \gamma(-\delta),\ve(1+a\delta) \big). \]

\begin{center}
\begin{picture}(400,180)
\put(180,10){Figure~6}

\thicklines
\put(80,100){\shade\circle{100}}
\put(200,100){\circle{60}}
\put(320,100){\shade\circle{20}}

\qbezier(80,150)(200,130)(320,110)
\put(80,100){\line(1,0){240}}
\qbezier(80,50)(200,70)(320,90)

\put(79,99){\rule{2pt}{2pt}}
\put(199,99){\rule{2pt}{2pt}}
\put(319,99){\rule{2pt}{2pt}}

\put(15,30){$A_-=B(\gamma(-\delta),\ve(1+a\delta))$}
\put(180,50){$B(x,\ve)$}
\put(250,60){$A_+=B(\gamma(\delta),\ve(1-a\delta))$}

\end{picture}
\end{center}

Note that $A_+=A_-=B(x,\ve)$ for $\delta=0$ and that
$(1/2)A_- +(1/2)A_+ =B(x,\ve)$.
We also observe that
\[ m(A_{\pm})
 =e^{-\psi(\gamma(\pm\delta))}c_n \ve^n (1\mp a\delta)^n+O(\ve^{n+1}), \]
where $c_n=\vol_n(B(0,1))$ and $O(\ve^{n+1})$ is independent of $\delta$.
Applying $(\ref{eq:BMN})$ to $A_{\pm}$ with $t=1/2$, we obtain
\begin{equation}\label{eq:A+-}
m\big( B(x,\ve) \big) \ge \frac{1}{2^N} \{ m(A_-)^{1/N}+m(A_+)^{1/N} \}^N.
\end{equation}
We know that $m(B(x,\ve))=e^{-\psi(x)}c_n\ve^n+O(\ve^{n+1})$.
In order to estimate the right-hand side, we calculate
\begin{align*}
&\frac{\del^2}{\del\delta^2}
 \Big[ e^{-\psi(\gamma(\delta))/N}(1-a\delta)^{n/N} \Big]\Big|_{\delta=0} \\
&=\bigg\{ -\frac{\Hess\psi(v,v)}{N}+\frac{\langle \nabla\psi,v \rangle^2}{N^2}
 +2\frac{\langle \nabla\psi,v \rangle}{N} \frac{n}{N}a
 +\frac{n}{N}\bigg( \frac{n}{N}-1 \bigg) a^2 \bigg\} e^{-\psi(x)/N} \\
&=\bigg\{ -\Hess\psi(v,v) +\frac{\langle \nabla\psi,v \rangle^2}{N-n}
 -\frac{n}{N(N-n)}\big( (N-n)a-\langle \nabla\psi,v \rangle \big)^2 \bigg\}
 \frac{e^{-\psi(x)/N}}{N}.
\end{align*}
Due to the choice of $a=\langle \nabla\psi(x),v \rangle/(N-n)$
(as the maximizer), we have
\[ \frac{\del^2}{\del\delta^2}
 \Big[ e^{-\psi(\gamma(\delta))/N}(1-a\delta)^{n/N} \Big]\Big|_{\delta=0}
 =\bigg\{ \frac{\langle \nabla\psi(x),v \rangle^2}{N-n} -\Hess\psi(v,v) \bigg\}
 \frac{e^{-\psi(x)/N}}{N}. \]
Thus we find, by the Taylor expansion of
$e^{-\psi(\gamma(\delta))/N} (1-a\delta)^{n/N}$ at $\delta=0$,
\begin{align*}
&\frac{m(A_-)^{1/N}+m(A_+)^{1/N}}{(c_n\ve^n)^{1/N}} \\
&= 2e^{-\psi(x)/N}
 -\bigg\{ \Hess\psi(v,v)-\frac{\langle \nabla\psi,v \rangle^2}{N-n} \bigg\}
 \frac{e^{-\psi(x)/N}}{N} \delta^2 +O(\delta^4) +O(\ve).
\end{align*}
Hence we obtain by letting $\ve$ go to zero in $(\ref{eq:A+-})$ that
\begin{align*}
e^{-\psi(x)} &\ge \frac{1}{2^N} \bigg[ 2e^{-\psi(x)/N}
 -\bigg\{ \Hess\psi(v,v)-\frac{\langle \nabla\psi,v \rangle^2}{N-n} \bigg\}
 \frac{e^{-\psi(x)/N}}{N} \delta^2 +O(\delta^4) \bigg]^N \\
&=e^{-\psi(x)} \bigg[ 1-\frac{1}{2}
 \bigg\{ \Hess\psi(v,v)-\frac{\langle \nabla\psi,v \rangle^2}{N-n} \bigg\} \delta^2 \bigg]
+O(\delta^4).
\end{align*}
Therefore we conclude
\[ \Hess\psi(v,v)-\frac{\langle \nabla\psi(x),v \rangle^2}{N-n} \ge 0. \]
$\qedd$
\end{proof}

Applying $(\ref{eq:BMN})$ to $A=\{x\}$, $B=B(x,R)$ and $t=r/R$ implies
\begin{equation}\label{eq:NBG}
\frac{m(B(x,R))}{m(B(x,r))} \le \bigg( \frac{R}{r} \bigg)^N
\end{equation}
for all $x \in \R^n$ and $0<r<R$.
Thus, compared with Theorem~\ref{th:BG}, $(\R^n,\|\cdot\|,m)$ satisfying
$(\ref{eq:NR})$ behaves like an `$N$-dimensional' space of nonnegative Ricci curvature
(see Theorem~\ref{th:CDBG} for more general theorem in terms of
the curvature-dimension condition).

\subsection{Characterizing lower Ricci curvature bounds}

Now we switch to the weighted Riemannian situation $(M,g,m)$,
where $m=e^{-\psi}\vol_g$ with $\psi \in C^{\infty}(M)$.
Ricci curvature controls $\vol_n$ as we saw in Section~\ref{sc:intro},
and Theorem \ref{th:BMN} suggests that the quantity
\[ \Hess\psi(v,v)-\frac{\langle \nabla\psi,v \rangle^2}{N-n} \]
has an essential information in controlling the effect of the weight.
Their combination indeed gives the weighted Ricci curvature as follows
(cf.\ \cite{BE}, \cite{Qi}, \cite{Lo}).

\begin{definition}[Weighted Ricci curvature]\label{df:wRic}
Given a unit tangent vector $v \in T_xM$ and $N \in [n,\infty]$,
the {\it weighted Ricci curvature} $\Ric_N(v)$ is defined by
\begin{enumerate}[(1)]
\item $\Ric_n(v):=\displaystyle \left\{
 \begin{array}{ll} \Ric(v)+\Hess\psi(v,v) & {\rm if}\ \langle \nabla\psi(x),v \rangle=0, \\
 -\infty & {\rm otherwise}; \end{array} \right.$

\item $\Ric_N(v):=\Ric(v) +\Hess\psi(v,v)
 -\displaystyle\frac{\langle \nabla\psi(x),v\rangle^2}{N-n}$ for $N \in (n,\infty)$;

\item $\Ric_{\infty}(v):=\Ric(v) +\Hess\psi(v,v)$.
\end{enumerate}
We say that $\Ric_N \ge K$ holds for $K \in \R$ if $\Ric_N(v) \ge K$
holds for all unit vectors $v \in TM$.
\end{definition}

Note that $\Ric_N \le \Ric_{N'}$ holds for $n \le N \le N'<\infty$.
$\Ric_{\infty}$ is also called the {\it Bakry-\'Emery tensor}.
If the weight is trivial in the sense that $\psi$ is constant,
then $\Ric_N$ coincides with $\Ric$ for all $N \in [n,\infty]$.
One of the most important examples possessing nontrivial weight
is the following.

\begin{example}[Euclidean spaces with log-concave measures]\label{ex:wRic}
Consider a weighted Euclidean space $(\R^n,\|\cdot\|,m)$ with
$m=e^{-\psi}\vol_n$, $\psi \in C^{\infty}(\R^n)$.
Then clearly $\Ric_{\infty}(v)=\Hess\psi(v,v)$, thus $\Ric_{\infty} \ge 0$
if $\psi$ is convex.
The most typical and important example satisfying $\Ric_{\infty} \ge K>0$ is
the Gausssian measure
\[ m=\bigg( \frac{K}{2\pi} \bigg)^{n/2} e^{-K\|x\|^2/2} \vol_n, \qquad
 \psi(x)=\frac{K}{2}\|x\|^2 +\frac{n}{2}\log\bigg( \frac{2\pi}{K} \bigg). \]
Note that $\Hess\psi \ge K$ holds independently of the dimension $n$.
\end{example}

Before stating the main theorem of the section,
we mention that optimal transport in a Riemannian manifold is described
in the same manner as the Euclidean spaces (due to \cite{Mc2}, \cite{CMS1}).
Given $\mu_0,\mu_1 \in \cP_c(M)$ with
$\mu_0=\rho_0 \vol_n \in \cP^{\ac}(M,\vol_n)$, there is a
{\it $(d^2/2)$-convex function} $f:M \lra \R$ such that
$\mu_t:=(\cF_t)_{\sharp}\mu_0$ with $\cF_t(x):=\exp_x[t\nabla f(x)]$,
$t \in [0,1]$, gives the unique minimal geodesic from $\mu_0$ to $\mu_1$.
We do not give the definition of $(d^2/2)$-convex functions,
but only remark that they are twice differentiable a.e.
Furthermore, the absolute continuity of $\mu_0$ implies that
$\mu_t$ is absolutely continuous for all $t \in [0,1)$,
so that we can set $\mu_t=\rho_tm$.
Since $\cF_t$ is differentiable $\mu_0$-a.e., we can consider the Jacobian
$\| (D\cF_t)_x \|$ (with respect to $\vol_n$)
which satisfies the Monge-Amp\`ere equation
\begin{equation}\label{eq:RMA}
\rho_0(x)=\rho_t\big( \cF_t(x) \big) \| (D\cF_t)_x \|
\end{equation}
for $\mu_0$-a.e.\ $x$.

We next introduce two entropy functionals.
Given $N \in [n,\infty)$ and an absolutely continuous probability measure
$\mu=\rho m \in \cP^{\ac}(M,m)$, we define the {\it R\'enyi entropy} as
\begin{equation}\label{eq:Ren}
S_N(\mu):=-\int_M \rho^{1-1/N} \,dm.
\end{equation}
We also define the {\it relative entropy} with respect to the reference measure $m$ by
\begin{equation}\label{eq:Ent}
\Ent_m(\mu):=\int_M \rho \log \rho \,dm.
\end{equation}
Note that $\Ent_m$ has the opposite sign to the Boltzmann entropy.
The domain of these functionals will be extended in the next section
($(\ref{eq:Ren'})$, $(\ref{eq:Ent'})$) to probability measures possibly with
nontrivial singular part.
In this section, however, we consider only absolutely continuous measures
for the sake of simplicity.
As any two points in $\cP^{\ac}_c(M,m)$ are connected
by a unique minimal geodesic contained in $\cP^{\ac}_c(M,m)$,
the convexity of $S_N$ and $\Ent_m$ in $\cP_c^{\ac}(M,m)$ makes sense.

Recall $(\ref{eq:beta})$ for the definition of the function $\beta^t_{K,N}$.
The following theorem is due to von Renesse, Sturm and many others,
see Further Reading for more details.

\begin{theorem}[A characterization of Ricci curvature bound]
\label{th:CD}
For a weighted Riemannian manifold $(M,g,m)$ with $m=e^{-\psi}\vol_g$,
$\psi \in C^{\infty}(M)$, we have $\Ric_N \ge K$ for some $K \in \R$
and $N \in [n,\infty)$ if and only if any pair of measures
$\mu_0=\rho_0 m, \mu_1=\rho_1 m \in \cP_c^{\ac}(M,m)$ satisfies
\begin{align}
S_N(\mu_t) &\le -(1-t)\int_{M \times M}
 \beta^{1-t}_{K,N}\big( d(x,y) \big)^{1/N} \rho_0(x)^{-1/N} \,d\pi(x,y) \nonumber\\
&\quad -t\int_{M \times M}
 \beta^t_{K,N}\big( d(x,y) \big)^{1/N} \rho_1(y)^{-1/N} \,d\pi(x,y) \label{eq:CDN}
\end{align}
for all $t \in (0,1)$,
where $(\mu_t)_{t \in [0,1]} \subset \cP_c^{\ac}(M,m)$ is the unique
minimal geodesic from $\mu_0$ to $\mu_1$
in the $L^2$-Wasserstein space $(\cP_2(M),d^W_2)$,
and $\pi$ is the unique optimal coupling of $\mu_0$ and $\mu_1$.

Similarly, $\Ric_{\infty} \ge K$ is equivalent to
\begin{equation}\label{eq:CD}
\Ent_m(\mu_t) \le
 (1-t)\Ent_m(\mu_0) +t\Ent_m(\mu_1) -\frac{K}{2}(1-t)td^W_2(\mu_0,\mu_1)^2.
\end{equation}
\end{theorem}

\begin{oproof}
We give a sketch of the proof for $N<\infty$ along the lines of
\cite{StII} and \cite{LVII}.
The case of $N=\infty$ goes along the essentially same line.

First, we assume $\Ric_N \ge K$.
Fix $\mu_0=\rho_0m, \mu_1=\rho_1m \in \cP_c^{\ac}(M,m)$ and take
a $(d^2/2)$-convex function $f:M \lra \R$ such that
$\cF_t(x):=\exp_x[t\nabla f(x)]$, $t \in [0,1]$, provides the unique minimal geodesic
$\mu_t=\rho_t m=(\cF_t)_{\sharp}\mu_0$ from $\mu_0$ to $\mu_1$.
Taking the weight $e^{-\psi}$ into account, we introduce the Jacobian
$\bJ^{\psi}_t(x):=e^{\psi(x)-\psi(\cF_t(x))}\| (D\cF_t)_x \|$ with respect to $m$
(like $\det_m$ in Theorem~\ref{th:BMN}).
Then it follows from the Monge-Amp\`ere equation $(\ref{eq:RMA})$
with respect to $\vol_n$ that
\begin{equation}\label{eq:MAp}
\rho_0(x)=\rho_t\big( \cF_t(x) \big) \bJ^{\psi}_t(x)
\end{equation}
for $\mu_0$-a.e.\ $x$ (i.e., the Monge-Amp\`ere equation with respect to $m$).

Now, the essential point is that optimal transport is performed along
geodesics $t \longmapsto \exp_x[t\nabla f(x)]=\cF_t(x)$.
Therefore its variational vector fields are Jacobi fields (recall $(\ref{eq:Ja})$),
and controlled by Ricci curvature.
Together with the weight control as in Theorem~\ref{th:BMN},
calculations somewhat similar to (but more involved than) Theorem~\ref{th:BG}
shows our key inequality
\begin{equation}\label{eq:J}
\bJ^{\psi}_t(x)^{1/N} \ge (1-t)\beta^{1-t}_{K,N}\big( d(x,\cF_1(x)) \big)^{1/N}
 +t\beta^t_{K,N}\big( d(x,\cF_1(x)) \big)^{1/N} \bJ^{\psi}_1(x)^{1/N}.
\end{equation}
This inequality can be thought of as an infinitesimal version of the Brunn-Minkowski
inequality (see $(\ref{eq:BMN})$ and Theorem~\ref{th:gBM}(i) as well).
As the change of variables formula and the Monge-Amp\`ere equation
$(\ref{eq:MAp})$ yield
\[ S_N(\mu_t) =-\int_M \rho_t(\cF_t)^{1-1/N} \bJ^{\psi}_t \,dm
 =-\int_M \bigg( \frac{\bJ^{\psi}_t}{\rho_0} \bigg)^{1/N} \,d\mu_0, \]
we obtain from $(\ref{eq:J})$ (and $(\ref{eq:MAp})$ again) that
\begin{align*}
S_N(\mu_t) &\le -(1-t)\int_M
 \frac{\beta^{1-t}_{K,N}(d(x,\cF_1(x)))^{1/N}}{\rho_0(x)^{1/N}} \,d\mu_0(x) \\
&\quad -t\int_M \frac{\beta^{t}_{K,N}(d(x,\cF_1(x)))^{1/N}}{\rho_1(\cF_1(x))^{1/N}}
 \,d\mu_0(x).
\end{align*}
This is the desired inequality $(\ref{eq:CDN})$,
for $\pi=(\Id_M \times \cF_1)_{\sharp}\mu_0$.

Second, we assume $(\ref{eq:CDN})$.
Then applying it to uniform distributions on balls (as in the proof of
Theorem~\ref{th:BMN}) shows $\Ric_N \ge K$.
More precisely, we use the generalized Brunn-Minkowski inequality
(Theorem \ref{th:gBM}) instead of $(\ref{eq:BMN})$.
$\qedd$
\end{oproof}

As $\beta^t_{0,N} \equiv 1$, the inequality $(\ref{eq:CDN})$ is simplified
into the convexity of $S_N$
\[ S_N(\mu_t) \le (1-t)S_N(\mu_0)+tS_N(\mu_1) \]
when $K=0$.
For $K \neq 0$, however, the $K$-convexity of $S_N$
\[ S_N(\mu_t) \le (1-t)S_N(\mu_0)+tS_N(\mu_1)
 -\frac{K}{2}(1-t)td^W_2(\mu_0,\mu_1)^2 \]
turns out uninteresting (see \cite[Theorem~1.3]{Stcon}).
This is a reason why we need to consider a more subtle inequality like $(\ref{eq:CDN})$.

Theorem \ref{th:CD} gives an answer to Question~\ref{Q:Ric},
for the conditions $(\ref{eq:CDN})$, $(\ref{eq:CD})$ are written in terms of
only distance and measure, without using the differentiable structure.
Then it is interesting to consider these conditions for general metric measure
spaces as `synthetic Ricci curvature bounds', and we should verify the stability.
We discuss them in the next section.

\begin{fread}
We refer to \cite{AGS}, \cite{Vi1} and \cite[Part~I]{Vi2} for the basics of optimal
transport theory and Wasserstein geometry.
McCann's~\cite{Mc2} fundamental result on the shape of optimal transport maps
is generalized to not necessarily compactly supported measures
in \cite{FF} and \cite{FG} (see also \cite[Chapter~10]{Vi2}).

See \cite{Ga} and \cite[Section~2.2]{Le} for the Brunn-Minkowski inequality
and related topics.
The Bakry-\'Emery tensor $\Ric_{\infty}$ was introduced in \cite{BE},
and its generalization $\Ric_N$ is due to Qian \cite{Qi}.
See also \cite{Lo} for geometric and topological applications,
\cite[Chapter~18]{Mo} and the references therein for minimal surface theory
in weighted manifolds (which are called {\it manifolds with density} there).

After McCann's~\cite{Mc1} pinoneering work on the convexity
of the relative entropy along geodesics in the Wasserstein space
(called the {\it displacement convexity}) over Euclidean spaces,
Cordero-Erausquin, McCann and Schmuckenschl\"ager~\cite{CMS1} first showed that
$\Ric \ge 0$ implies $(\ref{eq:CD})$ with $K=0$ in unweighted Riemannian manifolds.
They~\cite{CMS2} further proved that $\Ric_{\infty} \ge K$ implies
$(\ref{eq:CD})$ in the weighted situation.
Then Theorem~\ref{th:CD} is due to von Renesse and Sturm~\cite{vRS}, \cite{Stcon}
for $N=\infty$, and independently to Sturm~\cite{StI}, \cite{StII}
and Lott and Villani~\cite{LVI}, \cite{LVII} for $N<\infty$.

We comment on recent work on a variant of $(\ref{eq:CDN})$.
Studied in \cite{BaS1} is the following inequality (called the
{\it reduced curvature-dimension condition})
\begin{align}
S_N(\mu_t) &\le -(1-t)\int_{M \times M}
 \beta^{1-t}_{K,N+1}\big( d(x,y) \big)^{1/N} \rho_0(x)^{-1/N} \,d\pi(x,y)
 \nonumber\\
&\quad -t\int_{M \times M}
 \beta^t_{K,N+1}\big( d(x,y) \big)^{1/N} \rho_1(y)^{-1/N} \,d\pi(x,y).
 \label{eq:CD*}
\end{align}
Note the difference between
\[ t\beta^t_{K,N}(r)^{1/N}
 =t^{1/N} \bigg( \frac{\bs_{K,N}(tr)}{\bs_{K,N}(r)} \bigg)^{1-1/N}, \qquad
 t\beta^t_{K,N+1}(r)^{1/N} =\frac{\bs_{K,N+1}(tr)}{\bs_{K,N+1}(r)}. \]
We remark that $(\ref{eq:CD*})$ coincides with $(\ref{eq:CDN})$ when $K=0$,
and is weaker than $(\ref{eq:CDN})$ for general $K \neq 0$.
The condition $(\ref{eq:CD*})$ is also equivalent to $\Ric_N \ge K$
for Riemannian manifolds.
In the setting of metric measure spaces, $(\ref{eq:CD*})$ has
some advantages such as the tensorization and the localization properties
(see Subsection~\ref{ssc:mcp} for more details).
One drawback is that, as it is weaker than $(\ref{eq:CDN})$,
$(\ref{eq:CD*})$ derives slightly worse estimates than $(\ref{eq:CDN})$
(in the Bishop-Gromov volume comparison (Theorem~\ref{th:CDBG}),
the Bonnet-Myers diameter bound (Theorem~\ref{th:BMy}),
the Lichnerowicz inequality (Theorem~\ref{th:Lich}) etc.).
Nevertheless, such weaker estimates are sufficient for several topological applications.

See also \cite{Stcon} and \cite{OT} for the $K$-convexity of
generalized entropies (or free energies) and its characterization and applications.
It is discussed in \cite[Theorem~1.7]{Stcon} that there is a class of
functionals whose $K$-convexity is equivalent to $\Ric \ge K$ and $\dim \le N$
for unweighted Riemannian manifolds.
The choice of a functional is by no means unique, and it is unclear
how this observation relates to the curvature-dimension condition.
\end{fread}

\section{The curvature-dimension condition and stability}\label{sc:stab}

Motivated by Theorem~\ref{th:CD}, we introduce the curvature-dimension
condition for metric measure spaces and show that it is stable
under the measured Gromov-Hausdorff convergence.
In this and the next sections, $(X,d,m)$ will always be a metric measure
space in the sense of Section~\ref{sc:not}.

\subsection{The curvature-dimension condition}

We can regard the conditions $(\ref{eq:CDN})$, $(\ref{eq:CD})$ as convexity
estimates of the functionals $S_N$ and $\Ent_m$.
For the sake of consistency with the monotonicity of $\Ric_N$ in $N$
($\Ric_N \le \Ric_{N'}$ for $N \le N'$), we introduce
important classes of functionals (due to McCann~\cite{Mc1})
including $S_N$ and $\Ent_m$.

For $N \in [1,\infty)$, denote by $\DC_N$ ({\it displacement convexity class})
the set of continuous convex functions $U:[0,\infty) \lra \R$ such that $U(0)=0$
and that the function $\varphi(s)=s^N U(s^{-N})$ is convex on $(0,\infty)$.
Similarly, define $\DC_{\infty}$ as the set of continuous convex functions
$U:[0,\infty) \lra \R$ such that $U(0)=0$ and that
$\varphi(s)=e^s U(e^{-s})$ is convex on $\R$.
In both cases, the convexity of $U$ implies that $\varphi$ is non-increasing.
Observe the monotonicity, $\DC_{N'} \subset \DC_N$ holds for
$1 \le N \le N' \le \infty$.

For $\mu \in \cP(X)$, using its Lebesgue decomposition $\mu=\rho m +\mu^s$
into absolutely continuous and singular parts, we set
\[ U_m(\mu):=\int_X U(\rho) \,dm + U'(\infty) \mu^s(X),
 \qquad U'(\infty):=\lim_{r \to \infty} \frac{U(r)}{r}. \]
Note that $U'(\infty)$ indeed exists as $U(r)/r$ is non-decreasing.
In the case where $U'(\infty)=\infty$, we set $\infty \cdot 0:=0$ by convention.
The most important element of $\DC_N$ is the function $U(r)=Nr(1-r^{-1/N})$
which induces the R\'enyi entropy $(\ref{eq:Ren})$ in a slightly deformed form as
\begin{equation}\label{eq:Ren'}
U_m(\mu) =N\int_X \rho(1-\rho^{-1/N}) \,dm +N\mu^s(X)
 =N\bigg( 1-\int_X \rho^{1-1/N} \,dm \bigg).
\end{equation}
This extends $(\ref{eq:Ren})$ to whole $\cP(X)$.
Letting $N$ go to infinity, we have $U(r)=r\log r \in \DC_{\infty}$ as well as
the relative entropy (extending $(\ref{eq:Ent})$)
\begin{equation}\label{eq:Ent'}
U_m(\mu)=\int_X \rho \log\rho \,dm +\infty \cdot \mu^s(X).
\end{equation}

Let us denote by $\Gamma(X)$ the set of minimal geodesics $\gamma:[0,1] \lra X$
endowed with the distance
\[ d_{\Gamma(X)}(\gamma_1,\gamma_2)
 :=\sup_{t \in [0,1]}d_X\big( \gamma_1(t),\gamma_2(t) \big).\]
Define the evaluation map $e_t:\Gamma(X) \lra X$ for $t \in [0,1]$
as $e_t(\gamma):=\gamma(t)$, and note that this is $1$-Lipschitz.
A probability measure $\Pi \in \cP(\Gamma(X))$ is called a
{\it dynamical optimal transference plan} if the curve $\alpha(t):=(e_t)_{\sharp}\Pi$,
$t \in [0,1]$, is a minimal geodesic in $(\cP_2(X),d^W_2)$.
Then $\pi:=(e_0 \times e_1)_{\sharp}\Pi$ is an optimal coupling of
$\alpha(0)$ and $\alpha(1)$.
We remark that $\Pi$ is not uniquely determined by $\alpha$ and $\pi$,
that is to say, different plans $\Pi$ and $\Pi'$ could generate the same
minimal geodesic $\alpha$ and optimal coupling $\pi$.
If $(X,d)$ is locally compact (and hence proper), then any minimal geodesic in $\cP_2(X)$
is associated with a (not necessarily unique) dynamical optimal transference plan
(\cite[Proposition~2.10]{LVI}, \cite[Corollary~7.22]{Vi2}).

Now we are ready to present the precise definition of the curvature-dimension
condition in the form due to Lott and Villani (after Sturm and others,
see Further Reading of this and the previous sections).

\begin{definition}[The curvature-dimension condition]\label{df:CD}
Suppose that $m(B(x,r)) \in (0,\infty)$ holds for all $x \in X$ and $r \in (0,\infty)$.
For $K \in \R$ and $N \in (1,\infty]$, we say that a metric measure space $(X,d,m)$
satisfies the {\it curvature-dimension condition} $\CD(K,N)$ if,
for any $\mu_0=\rho_0 m+\mu_0^s$, $\mu_1=\rho_1 m+\mu_1^s \in \cP_b(X)$,
there exists a dynamical optimal transference plan $\Pi \in \cP(\Gamma(X))$
associated with a minimal geodesic $\alpha(t)=(e_t)_{\sharp}\Pi$, $t \in [0,1]$,
from $\mu_0$ to $\mu_1$ and an optimal coupling $\pi=(e_0 \times e_1)_{\sharp}\Pi$
of $\mu_0$ and $\mu_1$ such that we have
\begin{align}
U_m\big( \alpha(t) \big)
&\le (1-t)\int_{X \times X} \beta^{1-t}_{K,N}\big( d(x,y) \big)
 U\bigg( \frac{\rho_0(x)}{\beta^{1-t}_{K,N}(d(x,y))} \bigg) \,d\pi_x(y) dm(x) \nonumber\\
&\quad +t\int_{X \times X} \beta^t_{K,N}\big( d(x,y) \big)
 U\bigg( \frac{\rho_1(y)}{\beta^t_{K,N}(d(x,y))} \bigg) \,d\pi_y(x) dm(y) \nonumber\\
&\quad +U'(\infty)\{ (1-t)\mu_0^s(X)+t\mu_1^s(X) \} \label{eq:URic}
\end{align}
for all $U \in \DC_N$ and $t \in (0,1)$, where $\pi_x$ and $\pi_y$ denote
disintegrations of $\pi$ by $\mu_0$ and $\mu_1$, i.e.,
$d\pi(x,y)=d\pi_x(y)d\mu_0(x)=d\pi_y(x)d\mu_1(y)$.
\end{definition}

In the special case of $K=0$, as $\beta^t_{0,N} \equiv 1$,
the inequality $(\ref{eq:URic})$ means the convexity of $U_m$
\[ U_m\big( \alpha(t) \big) \le (1-t)U_m(\mu_0) +tU_m(\mu_1), \]
without referring to the optimal coupling $\pi$.
In the case where both $\mu_0$ and $\mu_1$ are absolutely continuous,
we have $d\pi(x,y)=\rho_0(x)d\pi_x(y)dm(x)=\rho_1(y)d\pi_y(x)dm(y)$
and hence $(\ref{eq:URic})$ is rewritten in a more symmetric form as
\begin{align}
U_m\big( \alpha(t) \big)
&\le (1-t)\int_{X \times X} \frac{\beta^{1-t}_{K,N}(d(x,y))}{\rho_0(x)}
 U\bigg( \frac{\rho_0(x)}{\beta^{1-t}_{K,N}(d(x,y))} \bigg) \,d\pi(x,y) \nonumber\\
&\quad +t\int_{X \times X} \frac{\beta^t_{K,N}(d(x,y))}{\rho_1(y)}
 U\bigg( \frac{\rho_1(y)}{\beta^t_{K,N}(d(x,y))} \bigg) \,d\pi(x,y). \label{eq:Uac}
\end{align}
Note that choosing $U(r)=Nr(1-r^{-1/N})$ and $U(r)=r\log r$ in $(\ref{eq:Uac})$
reduce to $(\ref{eq:CDN})$ and $(\ref{eq:CD})$, respectively.
We summarize remark on and the background of Definition~\ref{df:CD}.

\begin{remark}\label{rm:CD}
(a) It is easily checked that, if $(X,d,m)$ satisfies $\CD(K,N)$,
then the scaled metric measure space $(X,cd,c'm)$ for $c,c'>0$
satisfies $\CD(K/c^2,N)$.

(b) In Definition~\ref{df:CD}, to be precise, we need to impose the condition
\[ m\big( X \setminus B(x,\pi\sqrt{(N-1)/K}) \big)=0 \]
for all $x \in X$ if $K>0$ and $N<\infty$,
in order to stay inside the domain of $\beta^t_{K,N}$.
Nevertheless, this is always the case by virtue of the generalized
Bonnet-Myers theorem (Theorem~\ref{th:BMy}) below.

(c) Recall that $U(r)/r$ is non-decreasing, and observe that $\beta^t_{K,N}(r)$
is increasing in $K$ and decreasing in $N$.
Combining this with the monotonicity $\DC_{N'} \subset \DC_N$ for $N \le N'$
(and (b) above), we see that $\CD(K,N)$ implies $\CD(K',N')$ for all
$K' \le K$ and $N' \ge N$.
Therefore, in the condition $\CD(K,N)$, $K$ represents a lower bound of the Ricci curvature
and $N$ represents an upper bound of the dimension.

(d) The validity of $(\ref{eq:URic})$ along only `some' geodesic is essential
to establish the stability.
In fact, if we impose it for all geodesics between $\mu_0$ and $\mu_1$,
then it is not stable under convergence (in the sense of Theorem~\ref{th:stab}).
This is because, when a sequence $\{(X_i,d_i)\}_{i \in \N}$ converges to
the limit space $(X,d)$, there may be a geodesic in $X$ which can not be
represented as the limit of a sequence of geodesics in $X_i$.
Therefore the convexity along geodesics in $X_i$ does not necessarily imply
the convexity along all geodesics in $X$.
One typical example is a sequence of $\ell_p^n$-spaces as $p$ goes to $1$ or $\infty$.
Only straight lines are geodesics in $\ell_p^n$ with $1<p<\infty$,
while $\ell_1^n$ and $\ell_{\infty}^n$ have much more geodesics.
(See Figure~7, where $\gamma_i$ for all $i=0,\ldots,3$ are geodesic for $\ell^2_{\infty}$,
while only the straight line segment $\gamma_0$ is geodesic for $\ell^2_p$
with $1<p<\infty$.)
In fact, $\ell^n_p$ equipped with the Lebesgue measure satisfies $\CD(0,n)$
for all $1<p<\infty$ (Example~\ref{ex:Fins}(a)), but $\ell^n_1$ and $\ell^n_{\infty}$
do not satisfy $\CD(0,n)$.

\begin{center}
\begin{picture}(400,180)
\put(180,10){Figure~7}

\put(100,150){\line(1,0){200}}
\put(300,150){\line(0,-1){100}}
\put(100,150){\line(3,-1){150}}
\put(250,100){\line(1,-1){50}}
\qbezier(100,150)(100,50)(300,50)

\thicklines
\put(100,30){\vector(0,1){140}}
\put(80,50){\vector(1,0){240}}
\put(100,150){\line(2,-1){200}}

\put(85,145){$x$}
\put(295,35){$y$}
\put(170,100){$\gamma_0$}
\put(250,160){$\gamma_1$}
\put(210,120){$\gamma_2$}
\put(130,70){$\gamma_3$}

\end{picture}
\end{center}

(e) In Riemannian manifolds or, more generally, non-branching proper
metric measure spaces, we can reduce $(\ref{eq:URic})$ to a special case
from two aspects as follows.
If $(\ref{eq:Uac})$ holds for $U(r)=Nr(1-r^{-1/N})$ (or $U(r)=r\log r$ if $N=\infty$)
and all measures in $\cP^{\ac}_b(X,m)$ (and hence in $\cP^{\ac}_c(X,m)$)
with continuous densities,
then $(\ref{eq:URic})$ holds for all $U \in \DC_N$ and all measures in $\cP_b(X)$
(\cite[Proposition~4.2]{StII}, \cite[Proposition~3.21, Lemma~3.24]{LVI}).
In this sense, $(\ref{eq:CDN})$ and $(\ref{eq:CD})$ are essential among
the class of inequalities $(\ref{eq:URic})$.
A geodesic space is said to be non-branching if geodesics in it do not branch
(see Subsection~\ref{ssc:nb} for the precise definition).

(f) In Riemannian manifolds, $(\ref{eq:J})$ implies $(\ref{eq:Uac})$
for all $U \in \DC_N$.
Indeed, $\alpha(t)=\rho_t m$ is absolutely continuous and the change
of variables formula and the Monge-Amp\`ere equation $(\ref{eq:MAp})$ imply
\[ U_m\big( \alpha(t) \big) =\int_M U(\rho_t) \,dm
 =\int_M U\big( \rho_t(\cF_t) \big) \bJ^{\psi}_t \,dm
 =\int_M U\bigg( \frac{\rho_0}{\bJ^{\psi}_t} \bigg)
 \frac{\bJ^{\psi}_t}{\rho_0} \,d\mu_0. \]
For $N<\infty$, as $\varphi(s)=s^N U(s^{-N})$ is non-increasing and convex,
$(\ref{eq:J})$ yields
\begin{align*}
&U_m\big( \alpha(t) \big)
 \le \int_M \varphi\bigg(
 (1-t)\frac{\beta^{1-t}_{K,N}(d(x,\cF_1(x)))^{1/N}}{\rho_0(x)^{1/N}}
 +t\frac{\beta^t_{K,N}(d(x,\cF_1(x)))^{1/N}}{\rho_1(\cF_1(x))^{1/N}} \bigg)
 \,d\mu_0(x) \\
&\le \int_M \bigg\{ (1-t)\varphi\bigg(
 \frac{\beta^{1-t}_{K,N}(d(x,\cF_1(x)))^{1/N}}{\rho_0(x)^{1/N}} \bigg)
 +t\varphi\bigg(
 \frac{\beta^t_{K,N}(d(x,\cF_1(x)))^{1/N}}{\rho_1(\cF_1(x))^{1/N}} \bigg) \bigg\} \,d\mu_0(x).
\end{align*}
The case of $N=\infty$ is similar.
This means that the infinitesimal expression of $\CD(K,N)$
is always $(\ref{eq:J})$ whatever $U$ is, and various ways of
integration give rise to the definition of $\CD(K,N)$ involving $\DC_N$.

Although $\bJ^{\psi}_t$ relies on the differentiable structure of $M$,
it is possible to rewrite $(\ref{eq:J})$ through $(\ref{eq:MAp})$ as
\begin{align*}
\rho_t\big( \cF_t(x) \big)^{-1/N}
&\ge (1-t)\beta^{1-t}_{K,N}\big( d(x,\cF_1(x)) \big)^{1/N} \rho_0(x)^{-1/N} \\
&\quad +t\beta^t_{K,N}\big( d(x,\cF_1(x)) \big)^{1/N} \rho_1\big( \cF_1(x) \big)^{-1/N}
\end{align*}
(see \cite[Proposition~4.2(iv)]{StII}).
This makes sense in metric measure spaces, however, the integrated inequalities
(i.e., $(\ref{eq:URic})$, $(\ref{eq:Uac})$) are more convenient for verifying the stability.

(g) The role of the dynamical optimal transference plan $\Pi$ may seem unclear
in Definition~\ref{df:CD}, as only $\alpha$ and $\pi$ appear in $(\ref{eq:URic})$.
We use only $\alpha$ and $\pi$ also in applications in Section~\ref{sc:appl}.
As we shall see in Theorem~\ref{th:stab}, it is the stability in which
$\Pi$ plays a crucial role.
\end{remark}

\subsection{Stability and geometric background}

As we mentioned in Remark~\ref{rm:Alex}, one geometric motivation
behind the curvature-dimension condition is the theory of Alexandrov spaces.
That is to say, we would like to find a way of formulating and
investigating singular spaces of Ricci curvature bounded below
in some sense (recall Question~\ref{Q:Ric}).
Then, what kind of singular spaces should we consider?
Here comes into play another deep theory of the precompactness
with respect to the convergence of metric (measure) spaces.
Briefly speaking, the precompactness ensures that a sequence of
Riemannian manifolds with a uniform lower Ricci curvature bound
contains a convergent subsequence.
Such a limit space is not a manifold any more, but should inherits some properties.
In order to make use of the curvature-dimension condition in the limit,
we need to establish that it is preserved under the convergence
(actually, the limit of Alexandrov spaces is again an Alexandrov space.)

We say that a map $\varphi:Y \lra X$ between metric spaces is
{\it $\ve$-approximating} for $\ve>0$ if
\[ \big| d_X\big( \varphi(y),\varphi(z) \big) -d_Y(y,z) \big| \le \ve \]
holds for all $y,z \in Y$ and if $\overline{B(\varphi(Y),\ve)}=X$.

\begin{definition}[Measured Gromov-Hausdorff convergence]\label{df:mGH}
Consider a sequence of metric measure spaces $\{ (X_i,d_i,m_i) \}_{i \in \N}$ and
another metric measure space $(X,d,m)$.

\begin{enumerate}[(1)]
\item(Compact case)
Assume that $(X_i,d_i)$ for all $i \in \N$ and $(X,d)$ are compact.
We say that $\{ (X_i,d_i,m_i) \}_{i \in \N}$ converges to $(X,d,m)$
in the sense of the {\it measured Gromov-Hausdorff convergence}
if there are sequences of positive numbers $\{ \ve_i \}_{i \in \N}$
and Borel maps $\{ \varphi_i:X_i \lra X \}_{i \in \N}$ such that
$\lim_{i \to \infty}\ve_i=0$, $\varphi_i$ is an $\ve_i$-approximating map,
and that $(\varphi_i)_{\sharp}m_i$ weakly converges to $m$.

\item(Noncompact case)
Assume that $(X_i,d_i)$ for all $i \in \N$ and $(X,d)$ are proper,
and fix base points $x_i \in X_i$ and $x \in X$.
We say that $\{ (X_i,d_i,m_i,x_i) \}_{i \in \N}$ converges to $(X,d,m,x)$
in the sense of the {\it pointed measured Gromov-Hausdorff convergence}
if, for all $R>0$, $\{ (\overline{B(x_i,R)},d_i,m_i) \}_{i \in \N}$ converges to
$(\overline{B(x,R)},d,m)$ in the sense of the measured Gromov-Hausdorff
convergence (as in $(1)$ above).
\end{enumerate}
\end{definition}

If we consider only distance structures $(X_i,d_i)$ and $(X,d)$ and remove
the weak convergence condition on $\varphi_i$, then it is the ({\it pointed})
{\it Gromov-Hausdorff convergence} under which the lower sectional curvature bound
in the sense of Alexandrov is known to be preserved.
The following observation (\cite[Proposition~4.1]{LVI}) says that
the Gromov-Hausdorff convergence of a sequence of metric spaces
is propagated to the Wasserstein spaces over them.

\begin{proposition}\label{pr:WGH}
If a sequence of compact metric spaces $\{(X_i,d_i)\}_{i \in \N}$ converges to
a compact metric space $(X,d)$ in the sense of the Gromov-Hausdorff convergence,
then so does the sequence of Wasserstein spaces $\{(\cP(X_i),d^W_2)\}_{i \in \N}$
to $(\cP(X),d^W_2)$.

More precisely, $\ve_i$-approximating maps $\varphi_i:X_i \lra X$ give rise to
$\theta(\ve_i)$-approximating maps $(\varphi_i)_{\sharp}:\cP(X_i) \lra \cP(X)$
such that $\theta$ is a universal function satisfying
$\lim_{\ve \downarrow 0}\theta(\ve)=0$.
\end{proposition}

In the noncompact case, the pointed Gromov-Hausdorff convergence
of $\{(X_i,d_i,x_i)\}_{i \in \N}$ to $(X,d,x)$ similarly implies
the Gromov-Hausdorff convergence
of $\{(\cP(\overline{B(x_i,R)}),d^W_2)\}_{i \in \N}$
to $(\cP(\overline{B(x,R)}),d^W_2)$ for all $R>0$ (instead of the convergence of
$(\overline{B(\delta_{x_i},R)},d^W_2)$ to $(\overline{B(\delta_x,R)},d^W_2)$).

The following inspiring precompactness theorem is established by
Gromov~\cite[Section~5.A]{Gr} for the Gromov-Hausdorff convergence,
and extended by Fukaya~\cite{Fu} to the measured case.

\begin{theorem}[Gromov-Fukaya precompactness]\label{th:pre}
Let $\{(M_i,g_i,\vol_{g_i},x_i)\}_{i \in \N}$ be a sequence
of pointed Riemannian manifolds such that
\[ \Ric_{g_i} \ge K, \qquad \dim M_i \le N \]
uniformly hold for some $K \in \R$ and $N \in \N$.
Then it contains a subsequence which is convergent to some pointed
proper metric measure space $(X,d,m,x)$ in the sense of the pointed
measured Gromov-Hausdorff convergence.
\end{theorem}

To be more precise, we choose a complete space as the limit, and then
the properness follows from our hypotheses $\Ric \ge K$ and $\dim \le N$.
The key ingredient of the proof is the Bishop-Gromov volume comparison
(Theorem~\ref{th:BG}) from which we derive an upper bound of the
doubling constant $\sup_{x \in M,\, r \le R}\vol_g(B(x,2r))/\vol_g(B(x,r))$
for each $R \in (0,\infty)$.

By virtue of Theorem~\ref{th:pre}, starting from a sequence
of Riemannian manifolds with a uniform lower Ricci curvature bound,
we find the limit space of some subsequence.
Such a limit space is not a manifold any more,
but should have some inherited properties.
The stability of the curvature-dimension condition ensures that
we can use it for the investigation of these limit spaces.

We remark that the measures of balls $\vol_{g_i}(B(x_i,R))$ for $i \in \N$
have a uniform upper bound depending only on $K,N$ and $R$
thanks to Theorem~\ref{th:BG}.
However, it could tend to zero, and then we can not obtain any information on $(X,d,m)$.
Therefore we should take a scaling $c_i \vol_{g_i}$ with some appropriate
constant $c_i>1$.
It does not change anything because the weighted Ricci curvature
is invariant under scalings of the measure
(the weight function of $\tilde{m}=cm$ is $\psi_{\tilde{m}}=\psi_m-\log c$).
By the same reasoning, it is natural to assume that any bounded open ball
has a finite positive measure in the next theorem (see also Remark~\ref{rm:CD}(a)).

\begin{theorem}[Stability]\label{th:stab}
Assume that a sequence of pointed proper metric measure spaces
$\{ (X_i,d_i,m_i,x_i) \}_{i \in \N}$ uniformly satisfies $\CD(K,N)$
for some $K \in \R$ and $N \in (1,\infty]$ and that
it converges to a pointed proper metric measure space $(X,d,m,x)$
in the sense of the pointed measured Gromov-Hausdorff convergence.
If, moreover, $0<m(B(x,r))<\infty$ holds for all $x \in X$
and $r \in (0,\infty)$, then $(X,d,m)$ satisfies $\CD(K,N)$.
\end{theorem}

\begin{oproof}
The proof of stability goes as follows (along the lines of \cite{LVI}, \cite{LVII}).
First of all, as we consider only measures with bounded (and hence compact)
support in Definition~\ref{df:CD}, we can restrict ourselves to measures
with continuous density and compact support.
Indeed, it implies by approximation the general case
(\cite[Proposition~3.21, Lemma~3.24]{LVI}, see also Remark~\ref{rm:CD}(e)).
Given continuous measures $\mu=\rho m,\nu=\sigma m \in \cP_c^{\ac}(X,m)$
and $\ve_i$-approximating maps $\varphi_i:X_i \lra X$ as in Definition~\ref{df:mGH},
we consider
\[ \mu_i=\frac{\rho \circ \varphi_i}{\int_{X_i}\rho \circ \varphi_i \,dm_i} \cdot m_i,
 \quad
 \nu_i=\frac{\sigma \circ \varphi_i}{\int_{X_i}\sigma \circ \varphi_i \,dm_i} \cdot m_i
 \ \in \cP_c^{\ac}(X_i,m_i) \]
and take a dynamical optimal transference plan $\Pi_i \in \cP(\Gamma(X_i))$
from $\mu_i$ to $\nu_i$ satisfying $(\ref{eq:Uac})$.
Note that $(\varphi_i)_{\sharp}\mu_i$ and $(\varphi_i)_{\sharp}\nu_i$
weakly converge to $\mu$ and $\nu$, respectively, thanks to the continuity
of $\rho$ and $\sigma$.

By a compactness argument (\cite[Theorem~A.45]{LVII}), extracting a subsequence
if necessary, $\Pi_i$ converges to some dynamical transference plan
$\Pi \in \cP(\Gamma(X))$ from $\mu$ to $\nu$ such that, setting
\[ \alpha_i(t):=(e_t)_{\sharp}\Pi_i, \quad \alpha(t):=(e_t)_{\sharp}\Pi, \quad
 \pi_i:=(e_0 \times e_1)_{\sharp}\Pi_i, \quad \pi:=(e_0 \times e_1)_{\sharp}\Pi, \]
$(\varphi_i)_{\sharp}\alpha_i$ and $(\varphi_i \times \varphi_i)_{\sharp}\pi_i$
weakly converge to $\alpha$ and $\pi$, respectively.
Then it follows from Proposition~\ref{pr:WGH} that
$\alpha$ is a minimal geodesic from $\mu$ to $\nu$ and that
$\pi$ is an optimal coupling of $\mu$ and $\nu$.

On the one hand, the right-hand side of $(\ref{eq:Uac})$ for $\pi_i$
converges to that for $\pi$ by virtue of the continuous densities.
On the other hand, the monotonicity
\[ U_{(\varphi_i)_{\sharp}m_i}\big( (\varphi_i)_{\sharp}[\alpha_i(t)] \big)
 \le U_{m_i}\big( \alpha_i(t) \big) \]
and the lower semi-continuity
\[ U_m(\alpha(t)) \le \liminf_{i \to \infty}
 U_{(\varphi_i)_{\sharp}m_i}\big( (\varphi_i)_{\sharp}[\alpha_i(t)] \big) \]
hold true in general (\cite[Theorem~B.33]{LVI}).
Therefore we obtain $(\ref{eq:Uac})$ for $\Pi$ and complete the proof.
$\qedd$
\end{oproof}

\begin{fread}
The definition of the curvature-dimension condition is much indebted to
McCann's influential work~\cite{Mc1} introducing the important class of functions
$\DC_N$ as well as the displacement convexity along geodesics
in the Wasserstein space (see Further Reading in Section~\ref{sc:CD}).
Otto and Villani's work~\cite{OV} on the relation between such convexity
of the entropy and several functional inequalities was also inspiring.

The term `curvature-dimension condition' is used by Sturm~\cite{StI}, \cite{StII}
(and also in \cite{Vi2}) following Bakry and \'Emery's celebrated work \cite{BE}.
Sturm's condition requires that $(\ref{eq:URic})$ is satisfied for all absolutely
continuous measures and $U=S_{N'}$ for all $N' \in [N,\infty]$.
Lott and Villani~\cite{LVI}, \cite{LVII}, independently of Sturm,
introduced the condition as in Definition~\ref{th:CD}
and call it {\it $N$-Ricci curvature bounded from below by $K$}.
These conditions are equivalent in non-branching spaces
(see Remark~\ref{rm:CD}(e) and Subsection~\ref{ssc:nb}).
In locally compact non-branching spaces, it is also possible to extend $(\ref{eq:URic})$
from compactly supported measures to not necessarily compactly
supported measures (see \cite{FV}).

See \cite{Fu}, \cite[Chapter~3, Section~5.A]{Gr} and
\cite[Chapters~7, 8]{BBI} for the basics of (measured)
Gromov-Hausdorff convergence and for precompactness theorems.
The stability under the measured Gromov-Hausdorff convergence
we presented above is due to Lott and Villani~\cite{LVI}, \cite{LVII}.
Sturm~\cite{StI}, \cite{StII} also proved the stability with respect to a different,
his own notion of convegence induced from his ${\mathbf D}$-{\it distance}
between metric measure spaces.
Roughly speaking, the ${\mathbf D}$-distance takes couplings not only for
measures, but also for distances (see \cite{StI} for more details).

We also refer to celebrated work of Cheeger and Colding \cite{CC}
(mentioned in Further Reading of Section~\ref{sc:intro}) for related
geometric approach toward the investigation of limit spaces of Riemannian
manifolds of Ricci curvature bounded below.
Their strategy is to fully use the fact that it is the limit of Riemannian manifolds.
They reveal the detailed local structure of such limit spaces, however,
it also turns out that the limit spaces can have highly wild structures
(see a survey \cite{We} and the references therein).
We can not directly extend Cheeger and Colding's theory to metric measure
spaces with the curvature-dimension condition.
Their key tool is the Cheeger-Gromoll type splitting theorem,
but Banach spaces prevent us to apply it under the curvature-dimension
condition (see Subsection~\ref{ssc:Q}(C) for more details).
\end{fread}

\section{Geometric applications}\label{sc:appl}

Metric measure spaces satisfying the curvature-dimension condition $\CD(K,N)$ enjoy
many properties common to `spaces of dimension $\le N$ and Ricci curvature $\ge K$'.
To be more precise, though $N \in (1,\infty]$ is not necessarily an integer,
we will obtain estimates numerically extended to non-integer $N$.
Proofs based on optimal transport theory themselves are interesting and inspiring.
Although we concentrate on geometric applications in this article,
there are also many analytic applications including
the Talagrand inequality, logarithmic Sobolev inequality
(and hence the normal concentration of measures),
global Poincar\'e inequality and so forth (see \cite{LVI}, \cite{LVII}).

\subsection{Generalized Brunn-Minkowski inequality and applications}

Our first application is a generalization of the Brunn-Minkowski inequality
$(\ref{eq:BM})$, $(\ref{eq:BMN})$ to curved spaces.
This follows from the curvature-dimension condition $(\ref{eq:Uac})$
applied to $S_N$ and $\Ent_m$ (i.e., $(\ref{eq:CDN})$ and $(\ref{eq:CD})$)
between uniform distributions on two measurable sets.
In the particular case of $K=0$, we obtain the concavity of $m^{1/N}$ or $\log m$
as in $(\ref{eq:BM})$, $(\ref{eq:BMN})$.
Given two sets $A,B \subset X$ and $t \in (0,1)$, we denote by $Z_t(A,B)$
the set of points $\gamma(t)$ such that $\gamma:[0,1] \lra X$ is a minimal geodesic
with $\gamma(0) \in A$ and $\gamma(1) \in B$.
We remark that $Z_t(A,B)$ is not necessarily measurable regardless
the measurability of $A$ and $B$, however, it is not a problem
because $m$ is regular (see Remark~\ref{rm:mble}).

The following theorem is essentially contained in von Renesse and Sturm~\cite{vRS}
for $N=\infty$, and due to Sturm~\cite{StII} for $N<\infty$.
Again we will be implicitly indebted to Theorem~\ref{th:BMy}
that guarantees that $\diam X\le \pi\sqrt{(N-1)/K}$ if $K>0$ and $N<\infty$
(see Remark~\ref{rm:CD}(b)).
Figure~8 represents rough image of the theorem, $Z_{1/2}(A,B)$
has more measure in a positively curved space, and less measure
in a negatively curved space (compare this with Figure~1).

\begin{center}
\begin{picture}(400,200)
\put(180,10){Figure~8}
\put(90,30){$\Ric>0$}
\put(290,30){$\Ric<0$}

\thicklines
\put(40,120){\ellipse{30}{100}}
\put(110,120){\shade\ellipse{54}{120}}
\put(180,120){\ellipse{30}{100}}

\qbezier(40,170)(110,190)(180,170)
\qbezier(40,70)(110,50)(180,70)

\put(250,120){\ellipse{30}{100}}
\put(310,120){\shade\ellipse{20}{80}}
\put(370,120){\ellipse{30}{100}}

\qbezier(250,170)(310,150)(370,170)
\qbezier(250,70)(310,90)(370,70)

\put(35,115){$A$}
\put(175,115){$B$}
\put(83,115){$Z_{1/2}(A,B)$}
\put(245,115){$A$}
\put(365,115){$B$}
\put(285,60){$Z_{1/2}(A,B)$}

\end{picture}
\end{center}

\begin{theorem}[Generalized Brunn-Minkowski inequality]\label{th:gBM}
Take a metric measure space $(X,d,m)$ satisfying $\CD(K,N)$ and
two measurable sets $A,B \subset X$.
\begin{enumerate}[{\rm (i)}]
\item If $N \in (1,\infty)$, then we have
\begin{align*}
m\big( Z_t(A,B) \big)^{1/N}
&\ge (1-t)\inf_{x \in A,\, y \in B}\beta^{1-t}_{K,N}\big( d(x,y) \big)^{1/N}
 \cdot m(A)^{1/N} \\
&\quad +t\inf_{x \in A,\, y \in B}\beta^t_{K,N}\big( d(x,y) \big)^{1/N}
 \cdot m(B)^{1/N}
\end{align*}
for all $t \in (0,1)$.

\item If $N=\infty$ and $0<m(A),m(B)<\infty$, then we have
\begin{align*}
&\log m\big( Z_t(A,B) \big) \\
&\ge (1-t)\log m(A) +t\log m(B) +\frac{K}{2}(1-t)t
 d^W_2\bigg( \frac{\chi_A}{m(A)}m,\frac{\chi_B}{m(B)}m \bigg)^2
\end{align*}
for all $t \in (0,1)$.
\end{enumerate}
\end{theorem}

\begin{proof}
Similarly to Theorem~\ref{th:BMn}, we can assume that $A$ and $B$
are bounded and of positive measure.
Set
\[ \mu_0:=\frac{\chi_A}{m(A)} \cdot m, \quad
 \mu_1:=\frac{\chi_B}{m(B)} \cdot m, \quad
 \hat{\beta}^t:=\inf_{x \in A,\, y \in B}\beta^t_{K,N}\big( d(x,y) \big). \]

(i) We consider $U(r)=Nr(1-r^{-1/N})$ and recall from $(\ref{eq:Ren'})$ that,
for $\mu=\rho m+\mu^s$,
\[ U_m(\mu) =N\bigg( 1-\int_X \rho^{1-1/N} \,dm \bigg). \]
Hence it follows from $\CD(K,N)$ that there is a minimal geodesic
$\alpha:[0,1] \lra \cP(X)$ from $\mu_0$ to $\mu_1$ as well as an optimal
coupling $\pi$ such that, for all $t \in (0,1)$,
\begin{align*}
-\int_X \rho_t^{1-1/N} \,dm
&\le -(1-t)\int_{X \times X}
 \big\{ m(A)\beta^{1-t}_{K,N}\big( d(x,y) \big) \big\}^{1/N} \,d\pi(x,y) \\
&\quad -t\int_{X \times X}
 \big\{ m(B)\beta^t_{K,N}\big( d(x,y) \big) \big\}^{1/N} \,d\pi(x,y) \\
&\le -(1-t)(\hat{\beta}^{1-t})^{1/N} m(A)^{1/N}
 -t(\hat{\beta}^t)^{1/N} m(B)^{1/N},
\end{align*}
where we set $\alpha(t)=\rho_t m +\mu^s_t$.
Then the H\"older inequality yields
\[ \int_X \rho_t^{-1/N} \cdot \rho_t \,dm
 \le \bigg( \int_{\supp\rho_t} \rho_t^{-1} \cdot \rho_t \,dm \bigg)^{1/N}
 = m(\supp\rho_t)^{1/N} \le m\big( Z_t(A,B) \big)^{1/N}. \]
This completes the proof for $N<\infty$.

(ii) We argue similarly and obtain from $(\ref{eq:CD})$ that
\[ \Ent_m \big( \alpha(t) \big) \le
 -(1-t)\log m(A) -t\log m(B) -\frac{K}{2}(1-t)td^W_2(\mu_0,\mu_1)^2. \]
Note that, since $\Ent_m(\alpha(t))<\infty$, $\alpha(t)$ is absolutely continuous
and written as $\alpha(t)=\rho_t m$.
Furthermore, Jensen's inequality applied to the convex function
$s \longmapsto s\log s$ shows
\begin{align*}
\Ent\big( \alpha(t) \big)
&= m(\supp\rho_t) \int_{\supp\rho_t} \rho_t \log\rho_t \,\frac{dm}{m(\supp\rho_t)} \\
&\ge m(\supp\rho_t) \int_{\supp\rho_t} \rho_t \,\frac{dm}{m(\supp\rho_t)} \cdot
 \log\bigg( \int_{\supp\rho_t} \rho_t \,\frac{dm}{m(\supp\rho_t)} \bigg) \\
&=-\log m(\supp\rho_t) \ge -\log m\big( Z_t(A,B) \big).
\end{align*}
We complete the proof.
$\qedd$
\end{proof}

As a corollary, we find that $m$ has no atom unless $X$ consists of a single point.

\begin{corollary}\label{cr:atom}
If $(X,d,m)$ contains more than two points
and if it satisfies $\CD(K,N)$ with some $K \in \R$ and $N \in (1,\infty]$,
then any one point set $\{x\} \subset X$ has null measure.
\end{corollary}

\begin{proof}
It is sufficient to show the case of $N=\infty$ (due to Remark~\ref{rm:CD}(c)).
Put $A=\{x\}$ and assume that $m(\{x\})>0$ holds.
Take $r>0$ with $m(B(x,2r) \setminus B(x,r))>0$
(it is the case for small $r>0$), and note that
\[ Z_t\big( \{x\},B(x,2r) \setminus B(x,r)) \big)
 \subset B(x,2tr) \setminus B(x,tr) \]
for all $t \in (0,1)$.
Thus we apply Theorem~\ref{th:gBM}(ii) with $t=2^{-k}$, $k \in \N$, and find
\begin{align*}
&\log m\big( B(x,2^{1-k}r) \setminus B(x,2^{-k}r) \big) \\
&\ge (1-2^{-k})\log\big( m(\{x\}) \big)
 +2^{-k}\log\big( m(B(x,2r) \setminus B(x,r)) \big)
 -\frac{|K|}{2}(1-2^{-k})2^{-k}(2r)^2.
\end{align*}
Summing this up in $k \in \N$, we observe
\[ \sum_{k=1}^{\infty}
 \log m\big( B(x,2^{1-k}r) \setminus B(x,2^{-k}r) \big) =\infty. \]
This is a contradiction since we have
\begin{align*}
\sum_{k=1}^{\infty}
 \log m\big( B(x,2^{1-k}r) \setminus B(x,2^{-k}r) \big)
&\le \sum_{k=1}^{\infty} m\big( B(x,2^{1-k}r) \setminus B(x,2^{-k}r) \big) \\
&= m\big( B(x,2r) \setminus \{x\} \big) <\infty.
\end{align*}
$\qedd$
\end{proof}

Under $\CD(K,N)$ of the finite dimension $N<\infty$,
applying Theorem~\ref{th:gBM}(i) to thin annuli shows a generalization
of the Bishop-Gromov volume comparison theorem
(Theorem~\ref{th:BG}, see also $(\ref{eq:NBG})$).

\begin{theorem}[Generalized Bishop-Gromov volume comparison]\label{th:CDBG}
Suppose that a metric measure space $(X,d,m)$ satisfies $\CD(K,N)$ with
$K \in \R$ and $N \in (1,\infty)$.
Then we have
\[ \frac{m(B(x,R))}{m(B(x,r))} \le
 \frac{\int_0^R \bs_{K,N}(t)^{N-1}\,dt}{\int_0^r \bs_{K,N}(t)^{N-1}\,dt} \]
for all $x \in X$ and $0<r<R\ (\le \pi\sqrt{(N-1)/K}$ if $K>0)$.
\end{theorem}

\begin{proof}
The proof is essentially the same as the Riemannian case.
We apply Theorem~\ref{th:gBM}(i) to concentric thin annuli and obtain
an estimate corresponding to the Bishop area comparison of
concentric spheres $(\ref{eq:Bish2})$.
Then we take the sum and the limit (instead of integration),
and obtain the theorem.

We give more detailed calculation for thoroughness.
For any annulus $B(x,r_2) \setminus B(x,r_1)$ and $t \in (0,1)$,
Theorem~\ref{th:gBM}(i) and Corollary~\ref{cr:atom} yield that
\begin{align}
m\big( B(x,tr_2) \setminus B(x,tr_1) \big)
&\ge t^N \inf_{d \in [r_1,r_2]} \bigg( \frac{\bs_{K,N}(td)}{t\bs_{K,N}(d)} \bigg)^{N-1}
 m\big( B(x,r_2) \setminus B(x,r_1) \big) \nonumber\\
&\ge t \cdot \frac{\inf_{d \in [r_1,r_2]}\bs_{K,N}(td)^{N-1}}
 {\sup_{d \in [r_1,r_2]}\bs_{K,N}(d)^{N-1}} m\big( B(x,r_2) \setminus B(x,r_1) \big).
 \label{eq:CDBG}
\end{align}
This corresponds to $(\ref{eq:Bish2})$ in the Riemannian case.
Set $h(t):=\bs_{K,N}(t)^{N-1}$ for brevity,
and put $t_L:=(r/R)^{1/L}<1$ for $L \in \N$.
Applying $(\ref{eq:CDBG})$ to $r_1=t_Lr$, $r_2=r$ and $t=t_L^{l-1}$
for $l \in \N$, we have
\begin{align*}
m\big( B(x,r) \big)
&= \sum_{l=1}^{\infty} m\big( B(x,t_L^{l-1}r) \setminus B(x,t_L^l r) \big) \\
&\ge \bigg\{ \sum_{l=1}^{\infty} t_L^{l-1}
 \frac{\inf_{d \in [t_Lr,r]}h(t_L^{l-1}d)}{\sup_{d \in [t_Lr,r]}h(d)} \bigg\}
 m\big( B(x,r) \setminus B(x,t_L r) \big).
\end{align*}
We similarly deduce from $(\ref{eq:CDBG})$ with $r_1=t_L^{l-L}r$,
$r_2=t_L^{l-1-L}r$ and $t=t_L^{L-l+1}$ for $l=1,\ldots,L$ that
\begin{align*}
&m\big( B(x,r) \setminus B(x,t_L r) \big)
 \sum_{l=1}^L t_L^{l-1} \sup_{d \in [t_L^{l-L}r,t_L^{l-1-L}r]}h(d) \\
&\ge t_L^L \inf_{d \in [t_L r,r]}h(d)
 \sum_{l=1}^L m\big( B(x,t_L^{l-1-L}r) \setminus B(x,t_L^{l-L}r) \big) \\
&= \frac{r}{R} \inf_{d \in [t_L r,r]}h(d) \cdot m\big( B(x,R) \setminus B(x,r) \big).
\end{align*}
Combining these, we obtain
\begin{align*}
&m\big( B(x,r) \big) \cdot
 \sum_{l=1}^L (t_L^{l-1}-t_L^l)R \sup_{d \in [t_L R,R]}h(t_L^{l-1}d) \\
&= (1-t_L)R \cdot m\big( B(x,r) \big) \cdot
 \sum_{l=1}^L t_L^{l-1} \sup_{d \in [t_L^{l-L}r,t_L^{l-1-L}r]}h(d) \\
&\ge (1-t_L)R \cdot \bigg\{ \sum_{l=1}^{\infty} t_L^{l-1}
 \frac{\inf_{d \in [t_L r,r]}h(t_L^{l-1}d)}{\sup_{d \in [t_L r,r]}h(d)} \bigg\}
 \cdot \frac{r}{R} \inf_{d \in [t_L r,r]}h(d) \cdot m\big( B(x,R) \setminus B(x,r) \big) \\
&\ge m\big( B(x,R) \setminus B(x,r) \big) \cdot
 \frac{\inf_{d \in [t_L r,r]}h(d)}{\sup_{d \in [t_L r,r]}h(d)}
 \sum_{l=1}^{\infty} (t_L^{l-1}-t_L^l)r \inf_{d \in [t_L r,r]}h(t_L^{l-1}d).
\end{align*}
Letting $L$ diverge to infinity shows
\begin{equation}\label{eq:gBG}
m\big( B(x,r) \big) \int_r^R \bs_{K,N}(t)^{N-1} \,dt
 \ge m\big( B(x,R) \setminus B(x,r) \big) \int_0^r \bs_{K,N}(t)^{N-1} \,dt.
\end{equation}
This corresponds to $(\ref{eq:BG-})$ in the Riemannian case,
and the same calculation as the last step of the proof of Theorem~\ref{th:BG}
completes the proof.
$\qedd$
\end{proof}

Theorem~\ref{th:CDBG} shows that the doubling constant
$\sup_{x \in X,\, r \le R} m(B(x,2r))/m(B(x,r))$ is bounded
for each $R \in (0,\infty)$, therefore $X$ is proper
(see also the paragraph following Theorem~\ref{th:pre}).

\begin{corollary}\label{cr:prp}
Assume that $(X,d,m)$ satisfies $\CD(K,N)$ for some $K \in \R$ and
$N \in (1,\infty)$.
Then $(X,d)$ is proper.
\end{corollary}

Next we generalize the Bonnet-Myers diameter bound (Corollary~\ref{cr:BG}).
We remark that the proof below uses Theorem~\ref{th:gBM}(i) only for pairs of
a point $A=\{ x \}$ and a set $B \subset B(x,\pi\sqrt{(N-1)/K})$,
so that it is consistent with Remark~\ref{rm:CD}(b).
The following proof is due to \cite{Omcp}.

\begin{theorem}[Generalized Bonnet-Myers diameter bound]\label{th:BMy}
Suppose that a metric measure space $(X,d,m)$ satisfies $\CD(K,N)$
with $K>0$ and $N \in (1,\infty)$.
Then we have the following$:$

\begin{enumerate}[{\rm (i)}]
\item
It holds that $\diam X \le \pi\sqrt{(N-1)/K}$.

\item
Each $x \in X$ has at most one point of distance $\pi\sqrt{(N-1)/K}$ from $x$.
\end{enumerate}
\end{theorem}

\begin{proof}
It is enough to consider the case $K=N-1$ thanks to the
scaling property Remark~\ref{rm:CD}(a).

(i) Suppose that there is a pair of points $x,y \in X$ with $d(x,y)>\pi$,
set $\delta:=d(x,y)-\pi>0$ and take a minimal geodesic $\gamma:[0,\pi+\delta] \lra X$
from $x$ to $y$.
Choosing a different point on $\gamma$ if necessary,
we can assume $\delta<\pi/2$.
For $\ve \in (0,\delta)$, we apply Theorem~\ref{th:gBM}(i) between
$\{ \gamma(\delta+\ve) \}$ and $B(y,\ve)$ with $t=(\pi-\delta-\ve)/\pi$
and obtain
\[ \frac{m(Z_t(\{\gamma(\delta+\ve)\},B(y,\ve)))}{m(B(y,\ve))}
 \ge t^N \inf_{r \in (\pi-2\ve,\pi)} \bigg( \frac{\sin(tr)}{t\sin r} \bigg)^{N-1}
 =t \bigg( \frac{\sin(t(\pi-2\ve))}{\sin(\pi-2\ve)} \bigg)^{N-1}. \]
Then it follows from $t(\pi-2\ve) \le t\pi=\pi-\delta-\ve$ that
\begin{equation}\label{eq:BMi}
\frac{m(Z_t(\{\gamma(\delta+\ve)\},B(y,\ve)))}{m(B(y,\ve))}
 \ge \frac{\pi-\delta-\ve}{\pi} \bigg( \frac{\sin(\delta+\ve)}{\sin 2\ve} \bigg)^{N-1}
 \to \infty
\end{equation}
as $\ve$ tends to zero.
Given $z \in B(y,\ve)$, we take a minimal geodesic $\eta:[0,1] \lra X$
from $\gamma(\delta+\ve)$ to $z$ (see Figure~9),
and derive from the triangle inequality that
\[ d\big( \gamma(\delta+\ve),\eta(t) \big) =td\big( \gamma(\delta+\ve),z \big)
 <t\big\{ d\big( \gamma(\delta+\ve),y \big)+\ve \big\}
 =\pi-\delta-\ve. \]
Moreover, we deduce from $d(\gamma(\delta+\ve),z)<\pi$ that
\begin{align*}
d\big( x,\eta(t) \big) &\ge d(x,z)-d\big( z,\eta(t) \big)
 >\pi+\delta-\ve -(1-t)d\big( \gamma(\delta+\ve),z \big) \\
&>\pi+\delta-\ve -(1-t)\pi =\pi-2\ve.
\end{align*}
Thus we have
\begin{align*}
Z_t\big( \{\gamma(\delta+\ve)\},B(y,\ve) \big)
&\subset B\big( \gamma(\delta+\ve),\pi-\delta-\ve \big)
 \setminus B(x,\pi-2\ve) \\
&\subset B(x,\pi) \setminus B(x,\pi-2\ve).
\end{align*}
Combining this with $(\ref{eq:BMi})$, we conclude
\[ \lim_{\ve \downarrow 0}
 \frac{m(B(x,\pi) \setminus B(x,\pi-2\ve))}{m(B(y,\ve))}
 =\infty. \]

\setlength\unitlength{1pt}
\begin{center}
\begin{picture}(400,180)
\put(180,10){Figure~9}

\thicklines
\put(350,100){\shade\circle{40}}
\put(50,100){\line(1,0){300}}
\qbezier(100,100)(200,95)(350,90)
\qbezier(290,50)(305,100)(290,150)
\qbezier(315,35)(337,100)(315,165)

\put(38,98){$x$}
\put(80,110){$\gamma(\delta+\ve)$}
\put(220,110){$\gamma$}
\put(220,85){$\eta$}
\put(225,45){$S(x,\pi-2\ve)$}
\put(280,25){$S(\gamma(\delta+\ve),\pi-\delta-\ve)$}
\put(333,128){$B(y,\ve)$}
\put(352,85){$z$}
\put(303,78){$\eta(t)$}

\put(49,99){\rule{2pt}{2pt}}
\put(99,99){\rule{2pt}{2pt}}
\put(315,90){\rule{2pt}{2pt}}
\put(349,89){\rule{2pt}{2pt}}
\end{picture}
\end{center}

Furthermore, $(\ref{eq:gBG})$ and Theorem~\ref{th:CDBG} show that
\begin{align*}
&m\big( B(x,\pi) \setminus B(x,\pi-2\ve) \big)
 \le \frac{\int_{\pi-2\ve}^{\pi} \sin^{N-1}r \,dr}{\int_0^{\pi-2\ve} \sin^{N-1}r \,dr}
 m\big( B(x,\pi-2\ve) \big) \\
&=\frac{\int_0^{2\ve} \sin^{N-1}r \,dr}{\int_0^{\pi-2\ve} \sin^{N-1}r \,dr}
 m\big( B(x,\pi-2\ve) \big)
 \le m\big( B(x,2\ve) \big).
\end{align*}
Hence we have, again due to Theorem~\ref{th:CDBG} (with $K=0$),
\begin{equation}\label{eq:gBMy}
m\big( B(x,\pi) \setminus B(x,\pi-2\ve) \big) \le m\big( B(x,2\ve) \big)
 \le 2^N m\big( B(x,\ve) \big).
\end{equation}
Therefore we obtain $\lim_{\ve \downarrow 0}m(B(x,\ve))/m(B(y,\ve))=\infty$.
This is a contradiction because we can exchange the roles of $x$ and $y$.

(ii) We first see that $m(S(x,\pi))=0$, where $S(x,\pi):=\{ y \in X \,|\, d(x,y)=\pi \}$.
Given small $\ve>0$, take $\{ x_i \}_{i=1}^k \subset S(x,2\ve)$
such that $S(x,2\ve) \subset \bigcup_{i=1}^k B(x_i,2\ve)$
and $d(x_i,x_j) \ge 2\ve$ holds if $i \neq j$.
Then, for any $y \in S(x,\pi)$, there is some $x_i$ so that
$d(y,x_i)<(\pi-2\ve)+2\ve=\pi$, while $d(y,x_i) \ge \pi-2\ve$
holds in general.
Thus we have
\begin{align*}
m\big( S(x,\pi) \big)
&\le m\bigg( \bigcup_{i=1}^k B(x_i,\pi) \setminus B(x_i,\pi-2\ve) \bigg) \\
&\le \sum_{i=1}^k m\big( B(x_i,\pi) \setminus B(x_i,\pi-2\ve) \big).
\end{align*}
Then it follows from $(\ref{eq:gBMy})$ that
\begin{align*}
m\big( S(x,\pi) \big) &\le \sum_{i=1}^k m\big( B(x_i,2\ve) \big)
 \le 2^N \sum_{i=1}^k m\big( B(x_i,\ve) \big) \\
&=2^N m\bigg( \bigcup_{i=1}^k B(x_i,\ve) \bigg)
 \le 2^N m\big( B(x,3\ve) \big).
\end{align*}
Letting $\ve$ go to zero shows $m(S(x,\pi))=0$.

Now we suppose that there are two mutually distinct points
$y,z \in X$ such that $d(x,y)=d(x,z)=\pi$.
On the one hand, we derive from $(\ref{eq:gBMy})$ that
\[ m\big( B(x,r) \big) \ge m(B(x,\pi) \setminus B\big( x,\pi-r) \big) \]
for $r \in (0,\pi/2)$.
On the other hand, as $B(y,r) \subset X \setminus B(x,\pi-r)$ and
$m(S(x,\pi))=0$, we find
\[ m\big( B(y,r) \big) \le m\big( B(x,\pi) \setminus B(x,\pi-r) \big). \]
Hence we obtain $m(B(x,r)) \ge m(B(y,r))$ and
similarly $m(B(y,r)) \ge m(B(x,r))$.
This implies
\[ m\big( B(x,r) \big) =m\big( B(y,r) \big) =m\big( B(z,r) \big)
 =m\big( B(x,\pi) \setminus B(x,\pi-r) \big). \]
Then we have, for $\ve<d(y,z)/2$,
\begin{align*}
2m\big( B(x,\ve) \big)
&=m\big( B(y,\ve) \big) +m\big( B(z,\ve) \big) =m\big( B(y,\ve) \cup B(z,\ve) \big) \\
&\le m\big( B(x,\pi) \setminus B(x,\pi-\ve) \big) =m\big( B(x,\ve) \big).
\end{align*}
This is obviously a contradiction.
$\qedd$
\end{proof}

For $(X,d,m)$ satisfying $\CD(K,\infty)$ with $K>0$,
though $X$ is not necessarily bounded (see Example~\ref{ex:wRic}),
we can verify that $m(X)$ is finite (\cite[Theorem~4.26]{StI}).

\subsection{Maximal diameter}

In Riemannian geometry, it is well known that the maximal diameter
$\pi$ among (unweighted) Riemannian manifolds of Ricci curvature $\ge n-1$
is achieved only by the unit sphere $\Sph^n$.
In our case, however, orbifolds $\Sph^n/\Gamma$ can also have
the maximal diameter.
Hence what we can expect is a decomposition into a spherical suspension
(in some sense).
Due to the scaling property (Remark~\ref{rm:CD}(a)),
we consider only the case of $K=N-1>0$.
See \cite{Omcp2} for more precise discussion of the following theorem,
and Subsection~\ref{ssc:nb} for the definition of the non-branching property.

\begin{theorem}\label{th:sph}
Assume that $(X,d,m)$ is non-branching and satisfies $\CD(N-1,N)$
for some $N \in (1,\infty)$ as well as $\diam X=\pi$.
Then $(X,m)$ is the spherical suspension of some topological measure space.
\end{theorem}

\begin{oproof}
Fix $x_N,x_S \in X$ with $d(x_N,x_S)=\pi$.
Then it follows from $(\ref{eq:gBMy})$ that
\[ m\big( B(x_N,r) \big) +m\big( B(x_S,\pi-r) \big) =m(X) \]
for all $r \in (0,\pi)$.
This together with the non-branching property shows that,
for any $z \in X \setminus \{x_N,x_S\}$, there exists a unique minimal geodesic
from $x_N$ to $x_S$ passing through $z$.

Now we introduce the set $Y$ consisting of unit speed minimal geodesics
from $x_N$ to $x_S$, and equip it with the distance
\[ d_Y(\gamma_1,\gamma_2)
 :=\sup_{0\le t \le \pi} d_X\big( \gamma_1(t),\gamma_2(t) \big). \]
We consider $SY:=(Y \times [0,\pi])/\sim$, where
$(\gamma_1,t_1) \sim (\gamma_2,t_2)$ holds if $t_1=t_2=0$ or $t_1=t_2=\pi$.
We equip $SY$ with the topology naturally induced from $d_Y$.
Then the map $\Psi:SY \ni (\gamma,t) \longmapsto \gamma(t) \in X$
is well-defined and continuous.
Define the mesures $\nu$ on $Y$ and $\omega$ on $SY$ by
\begin{align*}
\nu(W) &:=\bigg\{ \int_0^{\pi} \sin^{N-1}t \,dt \bigg\}^{-1}
 m\big( \Psi(W \times [0,\pi]) \big), \\
d\omega &:=d\nu \times (\sin^{N-1}t \,dt).
\end{align*}
Then one can prove that $(SY,\omega)$ is regarded as the spherical suspension
of $(Y,\nu)$ as topological measure spaces, and that $\Psi:(SY,\omega) \lra (X,m)$
is homeomorphic and measure-preserving.
We use the non-branching property for the continuity of $\Psi^{-1}$.
$\qedd$
\end{oproof}

\subsection{Open questions}\label{ssc:Q}

We close the section with a list of open questions.

(A) (Beyond Theorem~\ref{th:sph})
There is room for improvement of Theorem~\ref{th:sph}:
Is the non-branching property necessary?
Can one say anything about the relation between the distances of $SY$ and $X$?
Does $(Y,d_Y,\nu)$ satisfy $\CD(N-2,N-1)$?

If $X$ is an $n$-dimensional Alexandrov space of curvature $\ge 1$
with $\diam X=\pi$, then it is isometric to the spherical suspension
of some $(n-1)$-dimensional Alexandrov space of curvature $\ge 1$.
It is generally difficult to derive something about distance from the
curvature-dimension condition.
We also do not know any counterexample.

(B) (Extremal case of Lichnerowicz inequality)
Related to Theorem~\ref{th:BMy}, we know the following
(\cite[Theorem~5.34]{LVII}).

\begin{theorem}[Generalized Lichnerowicz inequality]\label{th:Lich}
Assume that $(X,d,m)$ satisfies $\CD(K,N)$ for some $K>0$ and $N \in (1,\infty)$.
Then we have
\begin{equation}\label{eq:Lich}
\int_X f^2 \,dm \le \frac{N-1}{KN} \int_X |\nabla^- f|^2 \,dm
\end{equation}
for any Lipschitz function $f:X \lra \R$ with $\int_X f \,dm=0$.
\end{theorem}

Here $|\nabla^- f|$ is the {\it generalized gradient} of $f$ defined by
\[ |\nabla^- f|(x):=\limsup_{y \to x}\frac{\max\{ f(x)-f(y),0 \}}{d(x,y)}. \]
The proof is done via careful calculation using $(\ref{eq:Uac})$ for $S_N$ between
$m(X)^{-1} \cdot m$ and its perturbation $(1+\ve f)m(X)^{-1} \cdot m$.
The inequality $(\ref{eq:Lich})$ means that the lowest positive eigenvalue of the Laplacian
is larger than or equal to $KN/(N-1)$.
The constant $(N-1)/KN$ in $(\ref{eq:Lich})$ is sharp.
Moreover, in Riemannian geometry, it is known that the best constant
with $N=\dim M$ is achieved only by spheres.

In our general setting, it is not known whether the best constant is achieved
only by spaces of maximal diameter $\pi\sqrt{(N-1)/K}$.
If so, then Theorem~\ref{th:sph} provides us a decomposition into
a spherical suspension (for non-branching spaces).

(C) (Splitting)
In Riemannian geometry, Cheeger and Gromoll's~\cite{CG} celebrated theorem
asserts that, if a complete Riemannian manifold of nonnegative Ricci curvature
admits an isometric embedding of the real line $\R \hookrightarrow M$,
then $M$ isometrically splits off $\R$, namely $M$ is isometric to a product space
$M' \times \R$, where $M'$ again has the nonnegative Ricci curvature.
We can repeat this procedure if $M'$ contains a line.
This is an extremely deep theorem, and its generalization
is a key tool of Cheeger and Colding's seminal work~\cite{CC}
(see Further Reading of Section~\ref{sc:stab}).

Kuwae and Shioya~\cite{KS3} consider (weighted) Alexandrov spaces
of curvature $\ge -1$ with nonnegative Ricci curvature in terms of
the measure contraction property (see Subsection~\ref{ssc:mcp}).
They show that, if such an Alexandrov space contains an isometric copy
of the real line, then it splits off $\R$ as topological measure spaces
(compare this with Theorem~\ref{th:sph}).
This is recently strengthened into an isometric splitting by \cite{ZZ}
under a slightly stronger notion of Ricci curvature bound in terms of Petrunin's
second variation formula (\cite{Pegfa}).

For general metric measure spaces satisfying $\CD(0,N)$,
the isometric splitting is false because $n$-dimensional Banach spaces
satisfy $\CD(0,n)$ (Theorem~\ref{th:FCD}) and do not split in general.
The homeomorphic, measure-preserving splitting could be true,
but it is open even for non-branching spaces.

(D) (L\'evy-Gromov isoperimetric inequality)
Another challenging problem is to show (some appropriate variant of)
the L\'evy-Gromov isoperimetric inequality using optimal transport.
Most known proofs in the Riemannian case appeal to the deep existence and regularity theory
of minimal surfaces which can not be expected in singular spaces.

For instance, let us consider the {\it isoperimetric profile} $I_M:(0,m(M)) \lra (0,\infty)$
of a weighted Riemannian manifold $(M,g,m)$ with $m=e^{-\psi}\vol_g$, i.e.,
$I_M(V)$ is the least perimeter of sets with volume $V$.
Bayle~\cite{Bay} shows that the differential inequality
\begin{equation}\label{eq:IM}
(I_M^{N/(N-1)})'' \le -\frac{KN}{N-1}I_M^{1/(N-1)-1}
\end{equation}
holds if $\Ric_N \ge K$, which immediately implies the corresponding
{\it L\'evy-Gromov isoperimetric inequality}
\[ \frac{I_M(t \cdot m(M))}{m(M)} \ge I_{K,N}(t) \]
for $t \in [0,1]$, where $I_{K,N}$ is the isoperimetric profile of the $N$-dimensional
space form of constant sectional curvature $K/(N-1)$ equipped with the normalized
measure (extended to non-integer $N$ numerically).
The concavity estimate $(\ref{eq:IM})$ seems to be related to the Brunn-Minkowski
inequality, however, Bayle's proof of $(\ref{eq:IM})$ is based on the variational
formulas of minimal surfaces (see also \cite[Chapter~18]{Mo}).
More analytic approach could work in metric measure spaces,
but we need a new idea for it.

\begin{fread}
The generalized Brunn-Minkowski inequality (Theorem~\ref{th:gBM})
is essentially contained in the proof of \cite[Theorem~1.1]{vRS} for $N=\infty$,
and due to Sturm \cite{StII} for $N<\infty$
(also the Brascamp-Lieb inequality in \cite{CMS1} implies it
in the unweighted Riemannian situation with $N=n$).
It is used as a key tool in the proof of the derivation of the Ricci curvature bound
from the curvature-dimension condition (see Theorem~\ref{th:CD}).
Some more related interpolation inequalities can be found in \cite{CMS1},
these all were new even for Riemannian manifolds.

The relation between $\CD(K,\infty)$ and functional inequalities such as
the Talagrand, logarithmic Sobolev and the global Poincar\'e inequalities
are studied by Otto and Villani \cite{OV} and Lott and Villani \cite[Section~6]{LVI}.
We refer to \cite{Ol}, \cite{BoS} for related work on discrete spaces
(see also Subsection~\ref{ssc:Frm}(B)), and to \cite{Stcon}, \cite[Chapter~25]{Vi2},
\cite{OT} for the relation between variants of these functional inequalities
and the displacement convexity of generalized entropies.
Theorems~\ref{th:BMy} and \ref{th:sph} are due to \cite{Omcp} and \cite{Omcp2},
where the proof is given in terms of the measure contraction property
(see Subsection~\ref{ssc:mcp}).
\end{fread}

\section{The curvature-dimension condition in Finsler geometry}\label{sc:Fin}

In this section, we demonstrate that almost everything so far works well
also in the Finsler setting.
In fact, the equivalence between $\Ric_N \ge K$ and $\CD(K,N)$ is extended
by introducing an appropriate notion of the weighted Ricci curvature.
Then we explain why this is significant and discuss two potential applications.
We refer to \cite{BCS} and \cite{Shlec} for the fundamentals of
Finsler geometry, and the main reference of the section is \cite{Oint}.

\subsection{A brief introduction to Finsler geometry}

Let $M$ be an $n$-dimensional connected $C^{\infty}$-manifold.
Given a local coordinate $(x^i)_{i=1}^n$ on an open set $U \subset M$,
we always consider the coordinate $(x^i,v^i)_{i=1}^n$ on $TU$ given by
\[ v=\sum_{i=1}^n v^i \frac{\del}{\del x^i}\Big|_x \in T_xM. \]

\begin{definition}[Finsler structures]\label{df:Fins}
A {\it $C^{\infty}$-Finsler structure} is a nonnegative function
$F:TM \lra [0,\infty)$ satisfying the following three conditions:
\begin{enumerate}[(1)]
\item (Regularity)
$F$ is $C^{\infty}$ on $TM \setminus 0$,
where $0$ stands for the zero section;

\item (Positive homogeneity)
$F(\lambda v)=\lambda F(v)$ holds for all $v \in TM$ and $\lambda \ge 0$;

\item (Strong convexity)
Given a local coordinate $(x^i)_{i=1}^n$ on $U \subset M$,
the $n \times n$ matrix
\begin{equation}\label{eq:gij}
\big( g_{ij}(v) \big)_{i,j=1}^n
 := \bigg( \frac{1}{2} \frac{\del^2(F^2)}{\del v^i \del v^j}(v) \bigg)_{i,j=1}^n
\end{equation}
is positive-definite for all $v \in T_xM \setminus 0$, $x \in U$.
\end{enumerate}
\end{definition}

In other words, each $F|_{T_xM}$ is a {\it $C^{\infty}$-Minkowski norm}
(see Example~\ref{ex:Fins}(a) below for the precise definition)
and it varies $C^{\infty}$-smoothly also in the horizontal direction.
We remark that the homogeneity $(2)$ is imposed only in the positive direction,
so that $F(-v) \neq F(v)$ is allowed.
The positive-definite symmetric matrix $(g_{ij}(v))_{i,j=1}^n$ in $(\ref{eq:gij})$
defines the Riemannian structure $g_v$ on $T_xM$ through
\begin{equation}\label{eq:gv}
g_v\bigg( \sum_{i=1}^n v_1^i \frac{\del}{\del x^i}\Big|_x,
 \sum_{j=1}^n v_2^j \frac{\del}{\del x^j}\Big|_x \bigg)
 :=\sum_{i,j=1}^n g_{ij}(v) v_1^i v_2^j.
\end{equation}
Note that $F(v)^2=g_v(v,v)$.
If $F$ is coming from a Riemannian structure, then $g_v$ always coincides with
the original Riemannian metric.
In general, the inner product $g_v$ is regarded as the best approximation
of $F$ in the direction $v$.
More precisely, the unit spheres of $F$ and $g_v$ are tangent to each other at
$v/F(v)$ up to the second order
(that is possible thanks to the strong convexity, see Figure~10).

\begin{center}
\begin{picture}(400,200)
\put(180,10){Figure~10}

\put(200,30){\vector(0,1){160}}
\put(80,110){\vector(1,0){240}}

\qbezier(250,110)(245,80)(210,70)
\qbezier(210,70)(180,65)(140,80)
\qbezier(140,80)(110,90)(110,110)
\qbezier(110,110)(110,145)(130,160)
\qbezier(130,160)(180,195)(230,150)
\qbezier(230,150)(247,135)(250,110)

\thicklines
\put(200,110){\ellipse{100}{50}}
\put(200,110){\vector(1,0){50}}

\put(202,114){$v/F(v)$}
\put(140,140){$g_v(\cdot,\cdot)=1$}
\put(105,175){$F(\cdot)=1$}

\end{picture}
\end{center}

The {\it distance} between $x,y \in M$ is naturally defined by
\[ d(x,y):=\inf\bigg\{ \int_0^1 F(\dot{\gamma}) \,dt \,\Big|\,
 \gamma:[0,1] \lra M,\ C^1,\ \gamma(0)=x,\ \gamma(1)=y \bigg\}. \]
One remark is that the nonsymmetry $d(x,y) \neq d(y,x)$ may come up
as $F$ is only positively homogeneous.
Thus it is not totally correct to call $d$ a distance, it might be called
cost or action as $F$ is a sort of Lagrangian cost function.
Another remark is that the function $d(x,\cdot)^2$ is $C^2$ at the origin $x$
if and only if $F|_{T_xM}$ is Riemannian.
Indeed, the squared norm $|\cdot|^2$ of a Banach (or Minkowski) space
$(\R^n,|\cdot|)$ is $C^2$ at $0$ if and only if it is an inner product.

A $C^{\infty}$-curve $\gamma:[0,l] \lra M$ is called a {\it geodesic}
if it has constant speed ($F(\dot{\gamma}) \equiv c \in [0,\infty)$) and is
locally minimizing (with respect to $d$).
The reverse curve $\bar{\gamma}(t):=\gamma(l-t)$ is not
necessarily a geodesic.
We say that $(M,F)$ is {\it forward complete} if any geodesic
$\gamma:[0,\ve] \lra M$ is extended to a geodesic
$\overline{\gamma}:[0,\infty) \lra M$.
Then any two points $x,y \in M$ are connected by a minimal geodesic
from $x$ to $y$.

\subsection{Weighted Ricci curvature and the curvature-dimension condition}

We introduced distance and geodesics in a natural (metric geometric) way,
but the definition of curvature is more subtle.
The flag and Ricci curvatures on Finsler manifolds, corresponding to
the sectional and Ricci curvatures in Riemannian geometry,
are defined via some connection as in the Riemannian case.
The choice of connection is not unique in the Finsler setting,
nevertheless, all connections are known to give rise to the same curvature.
In these notes, however, we shall follow Shen's idea~\cite[Chapter~6]{Shlec}
of introducing the flag curvature using vector fields and corresponding
Riemannian structures (via $(\ref{eq:gv})$).
This intuitive description is not only geometrically understandable,
but also useful and inspiring.

Fix a unit vector $v \in T_xM \cap F^{-1}(1)$, and extend it to a
$C^{\infty}$-vector field $V$ on an open neighborhood $U$ of $x$
in such a way that every integral curve of $V$ is geodesic.
In particular, $V(\gamma(t))=\dot{\gamma}(t)$ along the geodesic
$\gamma:(-\ve,\ve) \lra M$ with $\dot{\gamma}(0)=v$.
Using $(\ref{eq:gv})$, we equip $U$ with the Riemannian structure $g_V$.
Then the {\it flag curvature} $\cK(v,w)$ of $v$ and a linearly independent
vector $w \in T_xM$ coincides with the sectional curvature with respect to $g_V$
of the $2$-plane $v \wedge w$ spanned by $v$ and $w$.
Similarly, the {\it Ricci curvature} $\Ric(v)$ of $v$ (with respect to $F$)
coincides with the Ricci curvature of $v$ with respect to $g_V$.
This contains the fact that $\cK(v,w)$ is independent of the choice
of the extension $V$ of $v$.
We remark that $\cK(v,w)$ depends not only on the {\it flag} $v \wedge w$,
but also on the choice of the {\it flagpole} $v$ in the flag $v \wedge w$.
In particular, $\cK(v,w) \neq \cK(w,v)$ may happen.

As for measure, on Finsler manifolds, there is no constructive measure
as good as the Riemannian volume measure.
Therefore, as the theory of weighted Riemannian manifolds,
we equip $(M,F)$ with an arbitrary positive $C^{\infty}$-measure $m$ on $M$.
Now, the weighted Ricci curvature is defined as follows (\cite{Oint}).
We extend given a unit vector $v \in T_xM$ to a $C^{\infty}$-vector field $V$ on
a neighborhood $U \ni x$ such that every integral curve is geodesic
(or it is sufficient to consider only the tangent vector field $\dot{\gamma}$
of the geodesic $\gamma:(-\ve,\ve) \lra M$ with $\dot{\gamma}(0)=v$),
and decompose $m$ as $m=e^{-\Psi(V)}\vol_{g_V}$ on $U$.
We remark that the weight $\Psi$ is not a function on $M$, but a function on
the unit tangent sphere bundle $SM \subset TM$.
For simplicity, we set
\begin{equation}\label{eq:Psi}
\del_v\Psi:=\frac{d(\Psi \circ \dot{\gamma})}{dt}(0), \qquad
 \del^2_v\Psi:=\frac{d^2(\Psi \circ \dot{\gamma})}{dt^2}(0).
\end{equation}

\begin{definition}[Weighted Ricci curvature of Finsler manifolds]\label{df:FRic}
For $N \in [n,\infty]$ and a unit vector $v \in T_xM$, we define
\begin{enumerate}[(1)]
\item $\Ric_n(v):=\displaystyle \left\{
 \begin{array}{ll} \Ric(v)+\del_v^2 \Psi & {\rm if}\ \del_v \Psi=0, \\
 -\infty & {\rm otherwise}; \end{array} \right.$

\item $\Ric_N(v):=\Ric(v) +\del_v^2 \Psi -\displaystyle\frac{(\del_v \Psi)^2}{N-n}$
for $N \in (n,\infty)$;

\item $\Ric_{\infty}(v):=\Ric(v) +\del_v^2 \Psi$.
\end{enumerate}
\end{definition}

In other words, $\Ric_N(v)$ of $F$ is $\Ric_N(v)$ of $g_V$
(recall Definition~\ref{df:wRic}), so that this curvature coincides with $\Ric_N$
in weighted Riemannian manifolds.
We remark that the quantity $\del_v\Psi$ coincides with Shen's $\bS$-curvature
(also called the {\it mean covariance} or {\it mean tangent curvature},
see \cite{Shvol}, \cite{Shlec}, \cite{Shsam}).
Therefore bounding $\Ric_n$ from below makes sense only when
the $\bS$-curvature vanishes everywhere.
This curvature enables us to extend Theorem~\ref{th:CD}
to the Finsler setting (\cite{Oint}).
Therefore all results in the theory of curvature-dimension condition
are applicable to general Finsler manifolds.

\begin{theorem}\label{th:FCD}
A forward complete Finsler manifold $(M,F,m)$ equipped with a positive
$C^{\infty}$-measure $m$ satisfies $\CD(K,N)$
for some $K \in \R$ and $N \in [n,\infty]$
if and only if $\Ric_N(v) \ge K$ holds for all unit vectors $v \in TM$.
\end{theorem}

We remark that, in the above theorem, the curvature-dimension condition
is appropriately extended to nonsymmetric distances.
The proof of Theorem~\ref{th:FCD} follows the same line as the Riemannian case,
however, we should be careful about nonsymmetric distance
and need some more extra discussion due to the fact that the squared
distance function $d(x,\cdot)^2$ is only $C^1$ at $x$.

We present several examples of Finsler manifolds.
The flag and Ricci curvatures are calculated in a number of situations,
while the weighted Ricci curvature is still relatively much less investigated.

\begin{example}\label{ex:Fins}
(a) (Banach/Minkowski spaces with Lebesgue measures)
A {\it Minkowski norm} $|\cdot|$ on $\R^n$ is a nonsymmetric generalization
of usual norms.
That is to say, $|\cdot|$ is a nonnegative function on $\R^n$ satisfying
the positive homogeneity $|\lambda v|=\lambda |v|$ for $v \in \R^n$
and $\lambda>0$; the convexity $|v+w| \le |v|+|w|$ for $v,w \in \R^n$;
and the positivity $|v|>0$ for $v \neq 0$.
Note that the unit ball of $|\cdot|$ is a convex (but not necessarily symmetric
to the origin) domain containing the origin in its interior
(see Figure~10, where $F$ is a Minkowski norm).

A Banach or Minkowski norm $|\cdot|$ which is $C^{\infty}$ on
$\R^n \setminus \{0\}$ induces a Finsler structure
in a natural way through the identification between $T_x\R^n$ and $\R^n$.
Then $(\R^n,|\cdot|,\vol_n)$ has the flat flag curvature.
Hence a Banach or Minkowski space $(\R^n,|\cdot|,\vol_n)$ satisfies $\CD(0,n)$
by Theorem~\ref{th:FCD} for $C^{\infty}$-norms,
and by Theorem~\ref{th:stab} via approximations for general norms.

(b) (Banach/Minkowski spaces with log-concave measures)
A Banach or Minkowski space $(\R^n,|\cdot|,m)$ equipped with
a measure $m=e^{-\psi}\vol_n$ such that $\psi$ is $K$-convex
with respect to $|\cdot|$ satisfies $\CD(K,\infty)$.
Here the {\it $K$-convexity} means that
\[ \psi\big( (1-t)x+ty \big) \le (1-t)\psi(x) +t\psi(y) -\frac{K}{2}(1-t)t|x-y|^2 \]
holds for all $x,y \in \R^n$ and $t \in [0,1]$.
This is equivalent to $\del_v^2 \psi \ge K$ (in the sense of $(\ref{eq:Psi})$)
if $|\cdot|$ and $\psi$ are $C^{\infty}$
(on $\R^n \setminus \{0\}$ and $\R^n$, respectively).
Hence $\CD(K,\infty)$ again follows from Theorem~\ref{th:FCD}
together with Theorem~\ref{th:stab}.

In particular, a Gaussian type space $(\R^n,|\cdot|,e^{-|\cdot|^2/2}\vol_n)$
satisfies $\CD(0,\infty)$ independently of $n$.
It also satisfies $\CD(K,\infty)$ for some $K>0$ if (and only if) it is
{\it $2$-uniformly convex} in the sense that $|\cdot|^2/2$ is
$C^{-2}$-convex for some $C \ge 1$ (see \cite{BCL} and \cite{Ouni}),
and then $K=C^{-2}$.
For instance, $\ell_p$-spaces with $p \in (1,2]$ are $2$-uniformly convex
with $C=1/\sqrt{p-1}$, and hence satisfies $\CD(p-1,\infty)$.
Compare this with Example~\ref{ex:wRic}.

(c) (Randers spaces)
A {\it Randers space} $(M,F)$ is a special kind of Finsler manifold such that
\[ F(v)=\sqrt{g(v,v)} +\beta(v) \]
for some Riemannian metric $g$ and a one-form $\beta$.
We suppose that $|\beta(v)|^2 <g(v,v)$ unless $v=0$,
then $F$ is indeed a Finsler structure.
Randers spaces are important in applications and reasonable
for concrete calculations.
In fact, we can see by calculation that $\bS(v)=\del_v\Psi \equiv 0$
holds if and only if $\beta$ is a Killing form of constant length
as well as $m$ is the Busemann-Hausdorff measure
(see \cite{Orand}, \cite[Section~7.3]{Shlec} for more details).
This means that there are many Finsler manifolds which do not admit any measures
of $\Ric_n \ge K>-\infty$, and then we must consider $\Ric_N$ for $N>n$.

(d) (Hilbert geometry)
Let $D \subset \R^n$ be a bounded open set with smooth boundary
such that its closure $\overline{D}$ is strictly convex.
Then the associated {\it Hilbert distance} is defined by
\[ d(x_1,x_2):=\log \bigg(
 \frac{\|x_1-x'_2\| \cdot \|x_2-x'_1\|}{\|x_1-x'_1\| \cdot \|x_2-x'_2\|} \bigg) \]
for distinct $x_1,x_2 \in D$, where $\|\cdot\|$ is the standard Euclidean norm
and $x'_1,x'_2$ are intersections of $\del D$ and the line passing through
$x_1,x_2$ such that $x'_i$ is on the side of $x_i$.
Hilbert geometry is known to be realized by a Finsler structure
with constant negative flag curvature.
However, it is still unclear if it carries a (natural) measure for which
the curvature-dimension condition holds.

(e) (Teichm\"uller space)
Teichm\"uller metric on Teichm\"uller space is one of the most famous
Finsler structures in differential geometry.
It is known to be complete, while the Weil-Petersson metric is incomplete
and Riemannian.
The author does not know any investigation concerned with
the curvature-dimension condition of Teichm\"uller space.
\end{example}

\subsection{Remarks and potential applications}\label{ssc:Frm}

Due to celebrated work of Cheeger and Colding~\cite{CC}, we know that
a (non-Hilbert) Banach space can not be the limit space of a sequence of
Riemannian manifolds (with respect to the measured Gromov-Hausdorff convergence)
with a uniform lower Ricci curvature bound.
Therefore the fact that Finsler manifolds satisfy the curvature-dimension condition
means that it is too weak to characterize limit spaces of Riemannian manifolds.
This should be compared with the following facts.

(I) A Banach space can be an Alexandrov space only if it happens to be
a Hilbert space (and then it has the nonnegative curvature);

(II) It is not known if all Alexandrov spaces $X$ of curvature $\ge k$
can be approximated by a sequence of Riemannian manifolds
$\{ M_i \}_{i \in \N}$ of curvature $\ge k'$.

We know that there are counterexamples to (II) if we impose the
non-collapsing condition $\dim M_i \equiv \dim X$ (see \cite{Ka}),
but the general situation admitting collapsing ($\dim M_i>\dim X$)
is still open and is one of the most important and challenging
questions in Alexandrov geometry.
Thus the curvature-dimension condition is not as good as the Alexandrov-Toponogov
triangle comparison condition from the purely Riemannian geometric viewpoint.

From a different viewpoint, Cheeger and Colding's observation means that
the family of Finsler spaces is properly much wider than the family of
Riemannian spaces.
Therefore the validity of the curvature-dimension condition for Finsler manifolds
opens the door to broader applications.
Here we mention two of them.

(A) (The geometry of Banach spaces)
Although their interested spaces are common to some extent, there is almost
no connection between the geometry of Banach spaces and Finsler geometry
(as far as the author knows).
We believe that our differential geometric technique would be useful in
the geometry of Banach spaces.
For instance, Theorem~\ref{th:FCD} (together with Theorem~\ref{th:stab})
could recover and generalize Gromov and Milman's normal concentration
of unit spheres in $2$-uniformly convex Banach spaces
(see \cite{GM2} and \cite[Section~2.2]{Le}).
To be precise, as an application of Theorem~\ref{th:FCD},
we know the normal concentration of Finsler manifolds such that
$\Ric_{\infty}$ goes to infinity (see \cite{Oint}, and \cite{GM1},
\cite[Section~2.2]{Le} for the Riemannian case).
This seems to imply the concentration of unit spheres mentioned above.

(B) (Approximations of graphs)
Generally speaking, Finsler spaces give much better approximations of graphs
than Riemannian spaces, when we impose a lower Ricci curvature bound.
For instance, Riemannian spaces into which the $\Z^n$-lattice
is nearly isometrically embedded should have very negative curvature,
while the $\Z^n$-lattice is isometrically embedded in flat $\ell^n_1$.
This kind of technique seems useful for investigating graphs with
Ricci curvature bounded below (in some sense), and provides a different point
of view on variants of the curvature-dimension condition for discrete spaces
(see, e.g., \cite{Ol}, \cite{BoS}).

\begin{fread}
We refer to \cite{BCS} and \cite{Shlec} for the fundamentals
of Finsler geometry and important examples.
The interpretation of the flag curvature using vector fields
can be found in \cite[Chapter~6]{Shlec}.
We also refer to \cite{Shvol} and \cite{Shsam} for the $\bS$-curvature
and its applications including a volume comparison theorem
different from Theorem~\ref{th:CDBG} (which has some topological
applications).
The $\bS$-curvature of Randers spaces and the characterization
of its vanishing (Example~\ref{ex:Fins}(c)) are studied
in \cite[Section~7.3]{Shlec} and \cite{Orand}.

Definition~\ref{df:FRic} and Theorem~\ref{th:FCD} are due to \cite{Oint},
while the weight function $\Psi$ on $SM$ has already been considered
in the definition of $\bS$-curvature.
See also \cite{OhS} for related work concerning heat flow on Finsler manifolds,
and \cite{Osurv} for a survey on these subjects.
The curvature-dimension condition $\CD(0,n)$ of Banach spaces
(Example~\ref{ex:Fins}(a)) is first demonstrated by Cordero-Erausquin
(see \cite[page 908]{Vi2}).
\end{fread}

\section{Related topics}\label{sc:rel}

We briefly comment on further related topics.

\subsection{Non-branching spaces}\label{ssc:nb}

We say that a geodesic space $(X,d)$ is {\it non-branching} if geodesics do not branch
in the sense that each quadruple of points $z,x_0,x_1,x_2 \in X$
with $d(x_0,x_1)=d(x_0,x_2)=2d(z,x_i)$ $(i=0,1,2)$ must satisfy $x_1=x_2$.
This is a quite useful property.
For instance, in such a space satisfying $\CD(K,N)$ for some $K$ and $N$,
a.e.\ $x \in X$ has unique minimal geodesic from $x$ to a.e.\ $y \in X$
(\cite[Lemma~4.1]{StII}).
Therefore we can localize the inequality $(\ref{eq:URic})$, and then $(\ref{eq:URic})$
for single $U=S_N$ implies that for all $U \in \DC_N$
(see Remark~\ref{rm:CD}(e), (f), \cite[Proposition~4.2]{StII}).
There are some more results known only in non-branching spaces (see, e.g.,
\cite[Section~4]{StII}, \cite{FV} and also Subsection~\ref{ssc:mcp} below).

Riemannian (or Finsler) manifolds and Alexandrov spaces are clearly non-branching.
However, as $n$-dimensional Banach and Minkowski spaces satisfy $\CD(0,n)$,
the curvature-dimension condition does not prevent the branching phenomenon.
One big open problem after Cheeger and Colding's work~\cite{CC} is
whether any limit space of Riemannian manifolds with a uniform
lower Ricci curvature bound is non-branching or not.

\subsection{Alexandrov spaces}\label{ssc:Al}

As was mentioned in Remark~\ref{rm:Alex}, Alexandrov spaces are metric spaces
whose sectional curvature is bounded from below in terms of the triangle
comparison property (see \cite{BGP}, \cite{OtS}, \cite[Chapters~4, 10]{BBI}
for more details).
One interesting fact is that a compact geodesic space $(X,d)$ is an Alexandrov space
of nonnegative curvature if and only if so is the Wasserstein space $(\cP(X),d^W_2)$
over it (\cite[Proposition~2.10]{StI}, \cite[Theorem~A.8]{LVI}).
This is a metric geometric explanation of Otto's formal calculation
of the sectional curvature of $(\cP_2(\R^n),d^W_2)$ (\cite{Ot}).
We remark that this relation can not be extended to positive or
negative curvature bounds.
In fact, if $(X,d)$ is not an Alexandrov space of nonnegative curvature,
then $(\cP(X),d^W_2)$ is not an Alexandrov space of curvature $\ge k$
even for negative $k$ (\cite[Proposition~2.10]{StI}).
Optimal transport in Alexandrov spaces is further studied in \cite{Be}, \cite{Ogra},
\cite{Sav}, \cite{Gi} and \cite{GO}.

Since the Ricci curvature is the trace of the sectional curvature,
it is natural to expect that Alexandrov spaces satisfy the curvature-dimension condition.
Petrunin \cite{Pecd} recently claims that it is indeed the case for $K=0$,
and is extended to the general case $K \neq 0$ by \cite{ZZ}.
They use the second variation formula in \cite{Pegfa} and
the gradient flow technique developed in \cite{PP} and \cite{Pesig},
instead of calculations as in Sections~\ref{sc:intro}, \ref{sc:CD}
involving Jacobi fields.

\subsection{The measure contraction property}\label{ssc:mcp}

For $K \in \R$ and $N \in (1,\infty)$, a metric measure space is said to satisfy
the {\it measure contraction property} $\MCP(K,N)$ if the Bishop inequality
$(\ref{eq:Bish})$ holds in an appropriate sense.
More precisely, $\MCP(K,N)$ for $(X,d,m)$ means that any $x \in X$
admits a measurable map $\Phi:X \lra \Gamma(X)$ satisfying
$e_0 \circ \Phi \equiv x$, $e_1 \circ \Phi =\Id_X$ and
\[ dm \ge (e_t \circ \Phi)_{\sharp}
 \Big( t^N \beta^t_{K,N}\big( d(x,y) \big) \,dm(y) \Big) \]
for all $t \in (0,1)$ as measures (compare this with Theorem~\ref{th:gBM}(i)).
As we mentioned in Further Reading in Section~\ref{sc:intro},
this kind of property was suggested in \cite[I, Appendix~2]{CC} and
\cite[Section 5.I]{Gr}, and systematically studied in \cite{Omcp},
\cite{Omcp2} and \cite[Sections~5, 6]{StII}.
Some variants have been also studied in \cite{KS1} and \cite{Staop}
before them.

$\MCP(K,N)$ can be regarded as the curvature-dimension condition $\CD(K,N)$ applied
only for each pair of a Dirac measure and a uniform distribution on a set,
and $\CD$ actually implies $\MCP$ in non-branching spaces.
It is known that Alexandrov spaces satisfy $\MCP$ (see \cite{Omcp}, \cite{KS2}).
For $n$-dimensional (unweighted) Riemannian manifolds,
$\MCP(K,n)$ is equivalent to $\Ric \ge K$, however,
$\MCP(K,N)$ with $N>n$ does not imply $\Ric \ge K$.
In fact, a sufficiently small ball in $\R^n$ equipped with the Lebesgue measure
satisfies $\MCP(1,n+1)$.
This is one drawback of $\MCP$.
On the other hand, an advantage of $\MCP$ is its simpleness,
there are several facts known for $\MCP$ and unknown for $\CD$.
We shall compare these properties in more details.

(A) (Product spaces ($L^2$-tensorization property))
If $(X_i,d_i,m_i)$ satisfies $\MCP(K_i,N_i)$ for $i=1,2$, then the product
metric measure space $(X_1 \times X_2,d_1 \times d_2,m_1 \times m_2)$
satisfies $\MCP(\min\{K_1,K_2\},N_1+N_2)$ (\cite{Omcp2}).
The analogous property for $\CD$ is known only for $\min\{K_1,K_2\}=0$
or $N_1+N_2=\infty$ in non-branching spaces (\cite{StI}).

Recently, Bacher and Sturm~\cite{BaS1} introduce a slightly weaker variant of $\CD$,
called the {\it reduced curvature-dimension condition} $\CD^*$
(recall $(\ref{eq:CD*})$ in Further Reading of Section~\ref{sc:CD}).
They show that $\CD^*$ enjoys the tensorization property
if the spaces in consideration are non-branching.

(B) (Euclidean cones)
If $(X,d,m)$ satisfies $\MCP(N-1,N)$, then its Euclidean cone
$(CX,d_{CX},m_{CX})$ defined by
\begin{align*}
CX &:= \big( X \times [0,\infty) \big)/\sim, \quad (x,0) \sim (y,0), \\
d_{CX}\big( (x,s),(y,t) \big) &:= \sqrt{s^2+t^2-2st\cos d(x,y)}, \\
dm_{CX} &:= dm \times (t^N \,dt)
\end{align*}
satisfies $\MCP(0,N+1)$ (\cite{Omcp2}).
This is recently established for the curvature-dimension condition by \cite{BaS2}
in the case where $(X,d,m)$ is Riemannian.

(C) (Local-to-global property)
Sturm~\cite{StI} shows that, if $(X,d,m)$ is non-branching and if every point in $X$
admits an open neighborhood on which $\CD(K,\infty)$ holds, then the whole space
$(X,d,m)$ globally satisfies $\CD(K,\infty)$.
In other words, $\CD(K,\infty)$ is a local condition as is the Ricci curvature
bound on a Riemannian manifold.
The same holds true also for $\CD(0,N)$ with $N<\infty$.
It is shown in \cite{BaS1} that $\CD^*(K,N)$ satisfies the local-to-global property
for general $K \in \R$ and $N\in (1,\infty)$, however, it is still open and
unclear if $\CD(K,N)$ for $K \neq 0$ and $N<\infty$ is a local condition.

In contrast, the local-to-global property is known to be false for $\MCP$.
As we mentioned above, sufficiently small balls in $\R^n$ satisfy
$\MCP(1,n+1)$, while the entire space $\R^n$ does not satisfy it by virtue of
the Bonnet-Myers diameter bound (Theorem~\ref{th:BMy}).


\begin{thebibliography}{CMS2}

\bibitem[AGS]{AGS}
L.~Ambrosio, N.~Gigli and G.~Savar\'e,
Gradient flows in metric spaces and in the space of probability measures,
Birkh\"auser Verlag, Basel, 2005.

\bibitem[BaS1]{BaS1}
K.~Bacher and K.-T.~Sturm,
Localization and tensorization properties of the curvature-dimension condition
for metric measure spaces,
J.\ Funct.\ Anal.\ {\bf 259} (2010), 28--56.

\bibitem[BaS2]{BaS2}
K.~Bacher and K.-T.~Sturm, Ricci bounds for Euclidean and spherical cones,
Preprint (2010). Available at {\sf arXiv:1003.2114}

\bibitem[BE]{BE}
D.~Bakry and M.~\'Emery, Diffusions hypercontractives (French),
S\'eminaire de probabilit\'es, XIX, 1983/84, 177--206, Lecture Notes in Math.\
{\bf 1123}, Springer, Berlin, 1985.

\bibitem[BCL]{BCL}
K.~Ball, E.~A.~Carlen and E.~H.~Lieb,
Sharp uniform convexity and smoothness inequalities for trace norms,
Invent.\ Math.\ {\bf 115} (1994), 463--482.

\bibitem[Bal]{Bal}
W.~Ballmann, Lectures on spaces of nonpositive curvature.
With an appendix by Misha Brin, Birkh\"auser Verlag, Basel, 1995.

\bibitem[BCS]{BCS}
D.~Bao, S.-S.~Chern and Z.~Shen, An introduction to Riemann-Finsler geometry,
Springer-Verlag, New York, 2000.

\bibitem[Bay]{Bay}
V.~Bayle, Propri\'et\'es de concavit\'e du profil isop\'erim\'etrique et applications (French),
Th\`ese de Doctorat, Institut Fourier, Universit\'e Joseph-Fourier, Grenoble, 2003.

\bibitem[Be]{Be}
J.~Bertrand, Existence and uniqueness of optimal maps on Alexandrov spaces,
Adv.\ Math.\ {\bf 219} (2008), 838--851.

\bibitem[BoS]{BoS}
A.-I.~Bonciocat and K.-T.~Sturm,
Mass transportation and rough curvature bounds for discrete spaces,
J.\ Funct.\ Anal.\ {\bf 256} (2009), 2944--2966.

\bibitem[Br]{Br}
Y.~Brenier, Polar factorization and monotone rearrangement of vector-valued functions,
Comm.\ Pure Appl.\ Math.\ {\bf 44} (1991), 375--417.

\bibitem[BBI]{BBI}
D.~Burago, Yu.~Burago and S.~Ivanov, A course in metric geometry,
American Mathematical Society, Providence, RI, 2001.

\bibitem[BGP]{BGP}
Yu.~Burago, M.~Gromov and G.~Perel'man,
A.~D.~Alexandrov spaces with curvatures bounded below,
Russian Math.\ Surveys {\bf 47} (1992), 1--58.

\bibitem[Ch]{Ch}
I.~Chavel, Riemannian geometry. A modern introduction. Second edition,
Cambridge University Press, Cambridge, 2006.

\bibitem[CC]{CC}
J.~Cheeger and T.~H.~Colding,
On the structure of spaces with Ricci curvature bounded below.~I, II, III,
J.\ Differential Geom.\ {\bf 46} (1997), 406--480; ibid.\ {\bf 54} (2000), 13--35; ibid.\ {\bf 54} (2000), 37--74.

\bibitem[CE]{CE}
J.~Cheeger and D.~G.~Ebin, Comparison theorems in Riemannian geometry,
Revised reprint of the 1975 original, AMS Chelsea Publishing, Providence, RI, 2008.

\bibitem[CG]{CG}
J.~Cheeger and D.~Gromoll,
The splitting theorem for manifolds of nonnegative Ricci curvature,
J.\ Differential Geometry {\bf 6} (1971/72), 119--128.

\bibitem[CMS1]{CMS1}
D.~Cordero-Erausquin, R.~J.~McCann and M.~Schmuckenschl\"ager,
A Riemannian interpolation inequality \`a la Borell, Brascamp and Lieb,
Invent.\ Math.\ {\bf 146} (2001), 219--257.

\bibitem[CMS2]{CMS2}
D.~Cordero-Erausquin, R.~J.~McCann and M.~Schmuckenschl\"ager,
Pr\'ekopa-Leindler type inequalities on Riemannian manifolds, Jacobi fields, and optimal transport,
Ann.\ Fac.\ Sci.\ Toulouse Math.\ (6) {\bf 15} (2006), 613--635.

\bibitem[FF]{FF}
A.~Fathi and A.~Figalli, Optimal transportation on non-compact manifolds,
Israel J.\ Math.\ {\bf 175} (2010), 1--59.

\bibitem[FG]{FG}
A.~Figalli and N.~Gigli,
Local semiconvexity of Kantorovich potentials on non-compact manifolds,
To appear in ESAIM Control Optim.\ Calc.\ Var.

\bibitem[FV]{FV}
A.~Figalli and C.~Villani, Strong displacement convexity on Riemannian manifolds,
Math.\ Z.\ {\bf 257} (2007), 251--259.

\bibitem[Fu]{Fu}
K.~Fukaya, Collapsing of Riemannian manifolds and eigenvalues of Laplace operator,
Invent.\ Math.\ {\bf 87} (1987), 517--547.

\bibitem[Ga]{Ga}
R.~J.~Gardner, The Brunn-Minkowski inequality,
Bull.\ Amer.\ Math.\ Soc.\ (N.S.) {\bf 39} (2002), 355--405.

\bibitem[Gi]{Gi}
N.~Gigli,
On the inverse implication of Brenier-McCann theorems and the
structure of $(\mathscr{P}_2(M),W_2)$, Preprint (2009).
Available at {\sf http://cvgmt.sns.it/people/gigli/}

\bibitem[GO]{GO}
N.~Gigli and S.~Ohta,
First variation formula in Wasserstein spaces over compact Alexandrov spaces,
Preprint (2010). Available at
{\sf http://www.math.kyoto-u.ac.jp/\textasciitilde sohta/}

\bibitem[Gr]{Gr}
M.~Gromov, Metric structures for Riemannian and non-Riemannian spaces,
Birkh\"auser, Boston, MA, 1999.

\bibitem[GM1]{GM1}
M.~Gromov and V.~D.~Milman, A topological application of the isoperimetric inequality,
Amer.\ J.\ Math.\ {\bf 105} (1983), 843--854.

\bibitem[GM2]{GM2}
M.~Gromov and V.~D.~Milman,
Generalization of the spherical isoperimetric inequality to uniformly convex Banach spaces,
Compositio Math.\ {\bf 62} (1987), 263--282.

\bibitem[Ka]{Ka}
V.~Kapovitch, Regularity of limits of noncollapsing sequences of manifolds,
Geom.\ Funct.\ Anal.\ {\bf 12} (2002), 121--137.

\bibitem[KS1]{KS1}
K.~Kuwae and T.~Shioya,
On generalized measure contraction property and energy functionals over Lipschitz maps,
ICPA98 (Hammamet). Potential Anal.\ {\bf 15} (2001), 105--121.

\bibitem[KS2]{KS2}
K.~Kuwae and T.~Shioya,
Infinitesimal Bishop-Gromov condition for Alexandrov spaces,
Adv.\ Stud.\ Pure Math.\ {\bf 57} (2010), 293--302.

\bibitem[KS3]{KS3}
K.~Kuwae and T.~Shioya,
A topological splitting theorem for weighted Alexandrov spaces, Preprint (2009).
Available at {\sf arXiv:0903.5150}

\bibitem[Le]{Le}
M.~Ledoux, The concentration of measure phenomenon,
American Mathematical Society, Providence, RI, 2001.

\bibitem[Lo1]{Lo}
J.~Lott, Some geometric properties of the Bakry-\'Emery-Ricci tensor, 
Comment.\ Math.\ Helv.\ {\bf 78} (2003), 865--883.

\bibitem[Lo2]{Losig}
J.~Lott, Optimal transport and Ricci curvature for metric-measure spaces,
Surveys in differential geometry {\bf XI}, 229--257, Int.\ Press, Somerville, MA, 2007.

\bibitem[LV1]{LVII}
J.~Lott and C.~Villani, Weak curvature conditions and functional inequalities,
J.\ Funct.\ Anal.\ {\bf 245} (2007), 311--333.

\bibitem[LV2]{LVI}
J.~Lott and C.~Villani, Ricci curvature for metric-measure spaces via optimal transport,
Ann.\ of Math.\ {\bf 169} (2009), 903--991.

\bibitem[Mc1]{Mc1}
R.~J.~McCann, A convexity principle for interacting gases,
Adv.\ Math.\ {\bf 128} (1997), 153--179.

\bibitem[Mc2]{Mc2}
R.~J.~McCann, Polar factorization of maps on Riemannian manifolds,
Geom.\ Funct.\ Anal.\ {\bf 11} (2001), 589--608.

\bibitem[Mo]{Mo}
F.~Morgan, Geometric measure theory. A beginner's guide. Fourth edition,
Elsevier/Academic Press, Amsterdam, 2009.

\bibitem[Oh1]{Omcp}
S.~Ohta, On the measure contraction property of metric measure spaces,
Comment.\ Math.\ Helv.\ {\bf 82} (2007), 805--828.

\bibitem[Oh2]{Omcp2}
S.~Ohta, Products, cones, and suspensions of spaces with the measure contraction property,
J.\ Lond.\ Math.\ Soc.\ $(2)$ {\bf 76} (2007), 225--236.

\bibitem[Oh3]{Ogra}
S.~Ohta, Gradient flows on Wasserstein spaces over compact Alexandrov spaces,
Amer.\ J.\ Math.\ {\bf 131} (2009), 475--516.

\bibitem[Oh4]{Ouni}
S.~Ohta, Uniform convexity and smoothness, and their applications in Finsler geometry,
Math.\ Ann.\ {\bf 343} (2009), 669--699.

\bibitem[Oh5]{Oint}
S.~Ohta, Finsler interpolation inequalities,
Calc.\ Var.\ Partial Differential Equations {\bf 36} (2009), 211--249.

\bibitem[Oh6]{Osurv}
S.~Ohta, Optimal transport and Ricci curvature in Finsler geometry,
Adv.\ Stud.\ Pure Math.\ {\bf 57} (2010), 323--342.

\bibitem[Oh7]{Orand}
S.~Ohta, Vanishing S-curvature of Randers spaces, Preprint (2009).
Available at {\sf arXiv:0909.1399}

\bibitem[OhS]{OhS}
S.~Ohta and K.-T.~Sturm, Heat flow on Finsler manifolds,
Comm.\ Pure Appl.\ Math.\ {\bf 62} (2009), 1386--1433.

\bibitem[OT]{OT}
S.~Ohta and A.~Takatsu, Displacement convexity of generalized entropies,
Preprint (2010). Available at {\sf arXiv:1005:1331}

\bibitem[Ol]{Ol}
Y.~Ollivier, Ricci curvature of Markov chains on metric spaces,
J.\ Funct.\ Anal.\ {\bf 256} (2009), 810--864.

\bibitem[OtS]{OtS}
Y.~Otsu, T.~Shioya, The Riemannian structure of Alexandrov spaces,
J.\ Differential Geom.\ {\bf 39} (1994), 629--658.

\bibitem[Ot]{Ot}
F.~Otto, The geometry of dissipative evolution equations: the porous medium equation,
Comm.\ Partial Differential Equations {\bf 26} (2001), 101--174.

\bibitem[OV]{OV}
F.~Otto and C.~Villani,
Generalization of an inequality by Talagrand and links with the logarithmic Sobolev inequality,
J.\ Funct.\ Anal.\ {\bf 173} (2000), 361--400.

\bibitem[PP]{PP}
G.~Perel'man and A.~Petrunin, Quasigeodesics and gradient curves in Alexandrov spaces,
Unpublished preprint (1994).

\bibitem[Pe1]{Pegfa}
A.~Petrunin,
Parallel transportation for Alexandrov space with curvature bounded below,
Geom.\ Funct.\ Anal.\ {\bf 8} (1998), 123--148.

\bibitem[Pe2]{Pesig}
A.~Petrunin, Semiconcave functions in Alexandrov's geometry,
Surveys in differential geometry.\ Vol.\ XI, 137--201,
Surv.\ Differ.\ Geom.\ {\bf 11}, Int.\ Press, Somerville, MA, 2007.

\bibitem[Pe3]{Pecd}
A.~Petrunin, Alexandrov meets Lott--Villani--Sturm, Preprint (2009).
Available at {\sf arXiv:1003.5948}

\bibitem[Qi]{Qi}
Z.~Qian, Estimates for weighted volumes and applications,
Quart.\ J.\ Math.\ Oxford Ser.\ (2) {\bf 48} (1997), 235--242.

\bibitem[vRS]{vRS}
M.-K.~von Renesse and K.-T.~Sturm,
Transport inequalities, gradient estimates, entropy and Ricci curvature,
Comm.\ Pure Appl.\ Math.\ {\bf 58} (2005), 923--940.

\bibitem[Sak]{Sak}
T.~Sakai, Riemannian geometry,
Translated from the 1992 Japanese original by the author.
Translations of Mathematical Monographs, {\bf 149}.
American Mathematical Society, Providence, RI, 1996.

\bibitem[Sav]{Sav}
G.~Savar\'e,
Gradient flows and diffusion semigroups in metric spaces under lower curvature bounds,
C.\ R.\ Math.\ Acad.\ Sci.\ Paris {\bf 345} (2007), 151--154.

\bibitem[Sh1]{Shvol}
Z.~Shen, Volume comparison and its applications in Riemann-Finsler geometry,
Adv.\ Math.\ {\bf 128} (1997), 306--328.

\bibitem[Sh2]{Shlec}
Z.~Shen, Lectures on Finsler geometry, World Scientific Publishing Co., Singapore, 2001.

\bibitem[Sh3]{Shsam}
Z.~Shen, Landsberg curvature, S-curvature and Riemann curvature,
A sampler of Riemann-Finsler geometry, 303--355,
Math.\ Sci.\ Res.\ Inst.\ Publ., {\bf 50}, Cambridge Univ.\ Press, Cambridge, 2004.

\bibitem[St1]{Staop}
K.-T.~Sturm, Diffusion processes and heat kernels on metric spaces,
Ann.\ Probab.\ {\bf 26} (1998), 1--55.

\bibitem[St2]{Stcon}
K.-T.~Sturm, Convex functionals of probability measures and nonlinear diffusions on manifolds,
J.\ Math.\ Pures Appl.\ {\bf 84} (2005), 149--168.

\bibitem[St3]{StI}
K.-T.~Sturm, On the geometry of metric measure spaces.~I,
Acta Math.\ {\bf 196} (2006), 65--131.

\bibitem[St4]{StII}
K.-T.~Sturm, On the geometry of metric measure spaces.~II,
Acta Math.\ {\bf 196} (2006), 133--177.

\bibitem[Vi1]{Vi1}
C.~Villani, Topics in optimal transportation, American Mathematical Society, Providence, RI, 2003.

\bibitem[Vi2]{Vi2}
C.~Villani, Optimal transport, old and new, Springer-Verlag, Berlin, 2009.

\bibitem[We]{We}
G.~Wei, Manifolds with a lower Ricci curvature bound,
Surveys in differential geometry {\bf XI}, 203--227, Int.\ Press, Somerville, MA, 2007.

\bibitem[ZZ]{ZZ}
H.-C.~Zhang and X.-P.~Zhu,
Ricci curvature on Alexandrov spaces and rigidity theorems,
Preprint (2009). Available at {\sf arXiv:0912.3190}

\end{thebibliography}
\end{document}